\documentclass[a4paper,11pt]{amsart}

\usepackage{amssymb,amsthm,amsmath,amscd,amsfonts,bbm,mathrsfs}

\usepackage[cmtip,all]{xy}

\theoremstyle{plain}
\newtheorem{Lemma}{Lemma}
\newtheorem{Thm}[Lemma]{Theorem}
\newtheorem*{Thm*}{Theorem}
\newtheorem{Prop}[Lemma]{Proposition}
\newtheorem{Cor}[Lemma]{Corollary}
\theoremstyle{definition}
\newtheorem{Defn}[Lemma]{Definition}
\newtheorem{Example}[Lemma]{Example}
\theoremstyle{remark}
\newtheorem{Remark}[Lemma]{Remark}

\newcommand{\BBB}{\mathscr{B}}
\newcommand{\CCC}{\mathscr{C}}
\newcommand{\FFF}{\mathcal{F}}
\newcommand{\OOO}{\mathcal{O}}
\newcommand{\PPP}{\mathscr{P}}
\newcommand{\UUU}{\mathscr{U}}
\newcommand{\VVV}{\mathscr{V}}
\newcommand{\WWW}{\mathscr{W}}
\newcommand{\XXX}{\mathscr{X}}
\newcommand{\YYY}{\mathscr{Y}}

\newcommand{\Fb}{\mathfrak{b}}
\newcommand{\Fm}{\mathfrak{m}}

\newcommand{\DD}{{\mathbb{D}}}
\newcommand{\FF}{{\mathbb{F}}}
\newcommand{\GG}{{\mathbb{G}}}
\newcommand{\QQ}{{\mathbb{Q}}}
\newcommand{\ZZ}{{\mathbb{Z}}}

\newcommand{\Ga}{{\mathbb{G}_a}}
\newcommand{\Gm}{{\mathbb{G}_m}}

\DeclareMathOperator{\Aff}{Aff}
\DeclareMathOperator{\Aut}{Aut}
\DeclareMathOperator{\BT}{BT}
\DeclareMathOperator{\Coker}{Coker}
\DeclareMathOperator{\Dieud}{Dieud}
\DeclareMathOperator{\Ext}{Ext}
\DeclareMathOperator{\Hom}{Hom}
\DeclareMathOperator{\GL}{GL}
\DeclareMathOperator{\Image}{Im}
\DeclareMathOperator{\Isom}{Isom}
\DeclareMathOperator{\Ker}{Ker}
\DeclareMathOperator{\Lie}{Lie}
\DeclareMathOperator{\Spec}{Spec}
\DeclareMathOperator{\Spf}{Spf}

\DeclareMathOperator{\uAut}{\underline{Aut}}
\DeclareMathOperator{\uHom}{\underline{Hom}}
\DeclareMathOperator{\uIsom}{\underline{Isom}}
\DeclareMathOperator{\uLie}{\underline{Lie}}

\DeclareMathOperator{\disp}{disp}
\DeclareMathOperator{\divv}{div}
\DeclareMathOperator{\form}{f.}
\DeclareMathOperator{\formgr}{f.g.}
\DeclareMathOperator{\grp}{grp}
\DeclareMathOperator{\id}{id}
\DeclareMathOperator{\ndisp}{n.disp}
\DeclareMathOperator{\per}{per}
\DeclareMathOperator{\predisp}{pre-disp}
\DeclareMathOperator{\red}{red}
\DeclareMathOperator{\univ}{univ}

\newcommand{\BTst}{\mathscr{B}\!\mathscr{T}}
\newcommand{\Disp}{\mathscr{D}isp}

\newcommand{\dispC}{{\disp}{}}
\newcommand{\predispC}{{\predisp}{}}
\newcommand{\ndispC}{{\ndisp}{}}
\newcommand{\fgC}{{\formgr}}
\newcommand{\pgrpC}{p\text{-}{\grp}}
\newcommand{\pdivC}{p\text{-}{\divv}{}}
\newcommand{\fpdivC}{{\form}\text{~$p$-}{\divv}}
\newcommand{\DieudC}{{\Dieud}{}}

\begin{document}

\title{Smoothness of the truncated display functor}
\author{Eike Lau}
\address{Fakult\"{a}t f\"{u}r Mathematik,
Universit\"{a}t Bielefeld, D-33501 Bielefeld}
\email{lau@math.uni-bielefeld.de}

\begin{abstract}
We show that to every $p$-divisible group over a $p$-adic ring 
one can associate a display by crystalline Dieudonn\'e theory. 
For an appropriate notion of truncated displays, this induces
a functor from truncated Barsotti-Tate groups to truncated
displays, which is a smooth morphism of smooth algebraic stacks.
As an application we obtain a new proof of the equivalence
between infinitesimal $p$-divisible groups and nilpotent
displays over $p$-adic rings, and a new proof of the equivalence
due to Berthelot and Gabber
between commutative finite flat group schemes of $p$-power order 
and Dieudonn\'e modules over perfect rings.
\end{abstract}

\maketitle

\section*{Introduction}

\renewcommand{\theLemma}{\Alph{Lemma}}

The notion of displays over $p$-adic rings arises naturally 
both in Cartier theory and in crystalline Dieudonn\'e theory. 

In Cartier theory, displays are a categorised form of
structure equations of Cartier modules of formal Lie 
groups. This is the original
perspective in \cite{Zink-Disp}. Passing from a
structure equation to the module corresponds to Zink's
functor $\BT$ from displays to formal Lie groups,
which induces an equivalence between nilpotent
displays and $p$-divisible formal groups by
\cite{Zink-Disp,Lau-Disp}. The theory includes
at its basis a description of the Dieudonn\'e crystal
of a $p$-divisible group $\BT(\PPP)$ in terms of the 
nilpotent display $\PPP$. We view this
as a passage from Cartier theory to crystalline
Dieudonn\'e theory.

On the other hand, let $G$ be a $p$-divisible group over 
a $p$-adic ring $R$, and let $D$ be the covariant Dieudonn\'e 
crystal of $G$. 
It is well-known that the Frobenius of $D$ restricted 
to the Hodge filtration is divisible by $p$.
If the ring of Witt vectors $W(R)$ has no $p$-torsion, this gives a
natural display structure on the value of $D$ on $W(R)$. 
We show that this construction extends in a 
unique way to a functor from $p$-divisible groups to displays 
over an arbitrary $p$-adic ring $R$
$$
\Phi_R:(\pdivC/R)\to(\dispC/R).
$$
The proof uses that the stacks of
truncated $p$-divisible groups are smooth algebraic stacks 
with smooth transition morphisms by \cite{Illusie-BT}, 
which implies that in a universal case the 
ring of Witt vectors has no $p$-torsion.

There is a natural notion of truncated displays over
rings of characteristic $p$. While a display is given
by an invertible matrix over $W(R)$ if a suitable basis
of the underlying module is fixed, a truncated display
is given by an invertible matrix over the truncated
Witt ring $W_n(R)$ for a similar choice of basis.
The functors $\Phi_R$ induce functors
from truncated Barsotti-Tate (BT) groups to truncated
displays of the same level
$$
\Phi_{n,R}:(\pdivC_n/R)\to(\dispC_n/R).
$$
For varying rings $R$
of characteristic $p$ they induce a morphism from
the stack of truncated BT groups of level $n$ to
the stack of truncated displays of level $n$, which
we denote by
$$
\phi_n:\BTst_n\to\Disp_n.
$$
The following is the central result of this article.

\begin{Thm}
\label{Th-A}
The morphism $\phi_n$ is a smooth morphism of smooth 
algebraic stacks over $\FF_p$, which is an equivalence
on geometric points.
\end{Thm}

Let us sketch the proof. The deformation theory of
nilpotent displays together with the crystalline
deformation theory of $p$-divisible groups implies
that the restriction of the functor $\Phi$ to
infinitesimal $p$-divisible groups is formally
\'etale in the sense that it induces an 
equivalence of infinitesimal deformations.
It follows that the smooth locus of $\phi_n$ 
contains all points of $\BTst_n$ that correspond
to infinitesimal groups; since the smooth locus
is open it must be all of $\BTst_n$.

For a truncated BT group $G$ we denote by 
$\uAut^o(G)$ the sheaf of automorphisms of $G$
which become trivial on the associated truncated display. 

\begin{Thm}
\label{Th-B}
Let $G_1$ and $G_2$ be truncated BT groups over a
ring $R$ of characteristic $p$ with associated 
truncated displays $\PPP_1$ and $\PPP_2$.
The group scheme $\uAut^o(G_i)$ is commutative, infinitesimal,
and finite flat over $R$. The natural morphism
$\uIsom(G_1,G_2)\to\uIsom(\PPP_1,\PPP_2)$ is
a torsor under $\uAut^o(G_i)$ for each $i$.
\end{Thm}

This is more or less a formal
consequence of Theorem \ref{Th-A}.
If $G$ is a truncated BT group of dimension $r$,
codimension $s$, and level $n$, one can show that
the degree of $\uAut^o(G)$ is equal to $p^{rsn}$. 
In particular, the functors $\Phi_{n,R}$ are
usually far from being an equivalence. 
The situation changes if one passes to the limit $\Phi_R$.
Namely, we have the following application of Theorems \ref{Th-A} 
and \ref{Th-B}:

\begin{Thm}
\label{Th-C}
For a $p$-adic ring $R$, the functor $\Phi_R$ induces
an equivalence between infinitesimal $p$-divisible groups
and nilpotent displays over $R$. 
\end{Thm}

Let us sketch the argument.
For a $p$-divisible group $G$ over a ring $R$ of 
characteristic $p$ the situation is controlled by 
the projective limit of finite flat group schemes
$$
\uAut^o(G)=\varprojlim_n\uAut^o(G[p^n]).
$$
If the group $G$ and its dual have a non-trivial 
\'etale part at some point of $\Spec R$, one can see 
directly that $\uAut^o(G)$ is non-trivial,
which explains the restriction to infinitesimal
groups in Theorem \ref{Th-C}. 
One has to show that $\uAut^o(G)$ is trivial
if $G$ is infinitesimal. 
If $\uAut^o(G)$ were non-trivial, the first
homology of its cotangent complex would be non-trivial, which
would contradict the fact that $\Phi$ is formally \'etale
for infinitesimal groups.

As a second application of Theorems \ref{Th-A} and \ref{Th-B}
we obtain an alternative proof of the following result of Gabber.

\begin{Thm}
\label{Th-D}
The category of $p$-divisible groups over a perfect ring $R$
of characteristic $p$ is equivalent to the category of
Dieudonn\'e modules over $R$.
\end{Thm}

As in the case of perfect fields,
a Dieudonn\'e module over $R$ is a 
projective $W(R)$-module $M$ of finite type
with a Frobenius-linear
endomorphism $F$ and a Frobenius$^{-1}$-linear endomorphism 
$V$ such that $FV=p=VF$. One deduces formally an
equivalence between commutative finite flat group schemes
of $p$-power order over $R$ and an appropriate category of
finite Dieudonn\'e modules. Over perfect valuation rings
this equivalence is proved by Berthelot \cite{Berthelot-Parfait},
and in general it is proved by Gabber by a reduction to the case
of valuation rings. 
Theorem \ref{Th-D} follows from Theorems \ref{Th-A} and \ref{Th-B}
since they show that the morphism $\phi_n$ is represented by a
morphism of groupoids of affine schemes which induces an isomorphism
of the perfect hulls.

Finally, we study the relation between the functors
$\Phi_R$ and $\BT_R$. One can form the composition
$$
(\pdivC/R)\xrightarrow{\Phi_R}
(\dispC/R)\xrightarrow{\BT_R}
(\text{formal groups}/R).
$$
Here both functors induce inverse equivalences
when restricted to formal $p$-divisible groups and
nilpotent displays.

\begin{Thm}
\label{Th-E}
For each $p$-divisible group $G$ over a $p$-adic ring $R$, 
the formal group $\BT_R(\Phi_R(G))$ is naturally isomorphic
to the formal completion $\hat G$.
\end{Thm} 

In other words, we have obtained a passage from crystalline
Dieudonn\'e theory to Cartier theory:
The natural display structure on the
Dieudonn\'e crystal of $G$, viewed as a 
structure equation of a Cartier module, gives the
Cartier module of $\hat G$.

\medskip

The author thanks Th.~Zink for many interesting discussions.
Section \ref{Subse-Erratum} contains an erratum to \cite{Lau-Disp}.
The author thanks O.~B\"ultel for pointing out this mistake.

\setcounter{tocdepth}{1}
\tableofcontents


\section{Preliminaries}

\numberwithin{Lemma}{section}
\numberwithin{equation}{section}

\subsection{Properties of ring homomorphisms}

All rings are commutative with a unit. 
Let $f:A\to B$ be a ring homomorphism.

We call $f$ \emph{ind-\'etale} (resp.\ \emph{ind-smooth}) 
if $B$ can be written as a filtered
direct limit of \'etale (resp.\ smooth) $A$-algebras.
In the ind-\'etale case the transition maps in the
filtered system are necessarily \'etale. 
We call $f$ an \emph{$\infty$-smooth} covering if there
is a sequence of faithfully flat smooth ring homomorphisms
$A=B_0\to B_1\to B_2\to\cdots$ with $B\cong\varinjlim B_i$.

We call $f$ \emph{reduced} if $f$ is flat and if the 
geometric fibres of $f$ are reduced.
This differs from EGA IV, 6.8.1, where in addition the 
fibres of $f$ are assumed to be noetherian. If $f$ is reduced,
then for each reduced $A$-algebra $A'$ the ring
$B\otimes_AA'$ is reduced. 
Every ind-smooth homomorphism is reduced.

Assume that $A$ and $B$ are noetherian. By the
Popescu desingularisation theorem, \cite[Thm.~2.5]{Popescu} 
and \cite{Swan-desingu}, $f$ is ind-smooth
if and only if $f$ is regular; recall that $f$ is
regular if $f$ is flat and if $B\otimes_AL$ is a
regular ring when $L$ is a finite extension of a residue 
field of a prime of $A$. 

Without noetherian hypothesis again,
we say that $f$ is \emph{quasi-\'etale} if the cotangent complex
$L_{B/A}$ is acyclic, and that $f$ is \emph{quasi-smooth} 
if the augmentation $L_{B/A}\to\Omega_{B/A}$ is a 
quasi-isomorphism and if $\Omega_{B/A}$ is a projective
$B$-module. Quasi-smooth implies formally smooth, and
quasi-\'etale implies formally \'etale; 
see \cite[III, Proposition 3.1.1]{Illusie-CC-I} and its proof.

\subsection{Affine algebraic stacks}
\label{Subse-aff-alg-stack}

Let $\Aff$ be the category of affine schemes.
Let $\XXX$ be a category which is fibered in groupoids
over $\Aff$. For 
a topology $\tau$ on $\Aff$, $\XXX$ is called a $\tau$-stack 
if $\tau$-descent is effective for $\XXX$. We call $\XXX$
an \emph{affine algebraic stack} if $\XXX$ is an fpqc stack,
if the diagonal morphism
$\XXX\to\XXX\times\XXX$ is representable affine, and if
there is an affine scheme $X$ with a faithfully flat 
morphism $X\to\XXX$, called a presentation of $\XXX$.
Equivalently, $\XXX$ is the fpqc stack associated to a 
flat groupoid of affine schemes.

Let $P$ be a property of ring homomorphisms which is stable
under base change. A representable affine morphism
of fpqc stacks is said to have the property $P$ if its
pull back to affine schemes has the property $P$. In particular,
one can demand that an affine algebraic stack has a 
presentation with the property $P$, called a $P$-presentation. 

Assume that the property $P$ is stable under composition and satisfies the following descent condition: If a composition of ring homomorphisms $v\circ u$ and $v$ have the property $P$ and if $v$ is faithfully flat, then $u$ has the property $P$. One example is $P$\,=\,reduced. Let $\XXX$ be an affine algebraic stack which has a $P$-presentation $X\to\XXX$. A morphism of affine algebraic stacks $\XXX\to\YYY$ is said to have the property $P$ if the composition $X\to\XXX\to\YYY$ has the property $P$. This does not depend on the $P$-presentation of $\XXX$.

Let $\XXX$ be an affine algebraic stack which has a reduced 
presentation $X\to\XXX$. We call $\XXX$ \emph{reduced} if $X$ 
is reduced; this does not depend on the reduced presentation.
In general, there is a maximal reduced closed
substack $\XXX_{\red}$ of $\XXX$. Indeed, the inverse images
of $X_{\red}$ under the two projections $X\times_{\XXX}X\to X$
are equal because they coincide with $(X\times_{\XXX}X)_{\red}$;
thus $X_{\red}$ descends to a substack of $\XXX$.

Assume that $\XXX$ is a locally noetherian Artin algebraic
stack and that $Y$ is a locally noetherian scheme. We call
a morphism $Y\to\XXX$ \emph{regular} if for a smooth
presentation $X\to\XXX$ the projection $Y\times_\XXX X\to X$
is regular. This is independent of the smooth presentation
of $\XXX$.

\subsection{The stack of $p$-divisible groups}
\label{Subse-stack-pdiv}

We fix a non-negative integer $h$.
Let $\BTst=\BTst^h$ be the stack of $p$-divisible groups
of height $h$, viewed as a fibered category over the
category of affine schemes. Thus for an affine scheme
$X$, $\BTst(X)$ is the category 
with $p$-divisible groups of height $h$ over $X$ as objects
and with isomorphisms of $p$-divisible groups as
morphisms. Similarly, for each non-negative integer $n$ let 
$\BTst_n=\BTst_n^h$ be the stack
of truncated Barsotti-Tate groups
of height $h$ and level $n$. This is an Artin algebraic
stack of finite type over $\ZZ$ with affine diagonal;
see \cite[Prop.~1.8]{Wedhorn} and \cite[Sec.~2]{Lau-Disp}.
The truncation morphisms 
$$
\tau_n:\BTst_{n+1}\to\BTst_n
$$ 
are smooth and surjective by 
\cite[Thm.~4.4 and Prop.~1.8]{Illusie-BT}.
Note that $\BTst_n$ has pure dimension zero
since the dense open substack $\BTst_n\times\Spec\QQ$
is the classifying space of the finite group
$\GL_h(\ZZ/p^n\ZZ)$.

\begin{Lemma}
\label{Le-BT-pres}
The fibered category $\BTst$ is an affine algebraic stack 
in the sense of section \ref{Subse-aff-alg-stack}. There is 
a presentation $\pi:X\to\BTst$ such that $\pi$ and the 
compositions $X\xrightarrow\pi\BTst\xrightarrow\tau\BTst_n$
for $n\ge 0$ are $\infty$-smooth coverings; in particular
$X\to\Spec\ZZ$ is an $\infty$-smooth covering.
\end{Lemma}

\begin{proof}
This follows from the properties of $\BTst_n$ and $\tau_n$, 
using that $\BTst$ is the projective limit of $\BTst_n$ 
for $n\to\infty$. More precisely,
the diagonal of $\BTst$ is representable affine 
because a projective limit of affine schemes is affine. 
We choose smooth presentations $\psi_n:Y_n\to\BTst_n$ with 
affine $Y_n$, and define recursively another sequence of smooth 
presentations $\pi_n:X_n\to\BTst_n$ with affine $X_n$ by 
$X_1=Y_1$ and $X_{n+1}=Y_{n+1}\times_{\BTst_n}X_n$. Let 
\begin{equation}
\label{Eq-BT-pres}
X=\varprojlim_n\,X_n=\varprojlim_n(X_n\times_{\BTst_n}\BTst)
\end{equation}
and let $\pi:X\to\BTst$ be the limit of the morphisms $\pi_n$.
The transition maps in the second system in \eqref{Eq-BT-pres}
are smooth and surjective because all $\psi_n$ are
smooth and surjective. Thus $\pi$ is presentation 
of $\BTst$ and an $\infty$-smooth covering.
The transition maps in the first system in \eqref{Eq-BT-pres}
are smooth
and surjective because the truncation morphisms $\tau_n$ are 
smooth and surjective too. Thus $X\to X_n\to\BTst_n$ is
an $\infty$-smooth covering.
\end{proof}

We refer to section \ref{Se-small-present} for presentations
of $\BTst\times\Spec\ZZ_p$ where the covering space is noetherian, 
and closer to $\BTst$ in some sense.

\subsection{Newton stratification}
\label{Subse-Newton}

In the following we write
$$
\overline\BTst=\BTst\times\Spec\FF_p;\qquad
\overline{\BTst_n}=\BTst_n\times\Spec\FF_p.
$$
We call a Newton polygon of height $h$ a polygon that 
appears as the Newton polygon of a $p$-divisible group of 
height $h$.

\begin{Lemma}
\label{Le-Newton}
For each Newton polygon $\nu$ of height $h$ there is a unique
reduced closed substack $\BTst_\nu$ of $\overline\BTst$
such that the geometric points of $\BTst_\nu$ are the
$p$-divisible groups with Newton polygon $\preceq\nu$. 
\end{Lemma}

\begin{proof}
We consider a reduced presentation 
$X\to\overline\BTst$ with affine $X$, 
defined by the $p$-divisible group $G$ over $X$.
The points of $X$ where $G$ has Newton polygon $\preceq\nu$
form a closed subset of $X$; see \cite[Thm.~2.3.1]{Katz-Slope}.
The corresponding reduced subscheme $X_{\nu}$ of $X$ descends
to a reduced substack of $\BTst$ because 
the inverse images of $X_\nu$ under the two projections 
$X\times_{\BTst}X\to X$ are
reduced and coincide on geometric points, so they are equal.
\end{proof}

By a well-known boundedness principle, there is an integer
$N$ depending on $h$ such that the Newton polygon of
a $p$-divisible group $G$ of height $h$ is determined by
its truncation $G[p^N]$. 

\begin{Lemma}
\label{Le-Newton-trunc}
For $n\ge N$ there is a unique reduced closed substack 
$\BTst_{n,\nu}$ of $\overline{\BTst_n}$ such that we have 
a Cartesian diagram
$$
\xymatrix@M+0.2em{
\BTst_\nu \ar[r] \ar[d] & 
\BTst \ar[d]^\tau \\
\BTst_{n,\nu} \ar[r] & {\BTst_n}
}
$$
where $\tau$ is the truncation.
In particular, the closed immersion $\BTst_{\nu}\to\BTst$
is a morphism of finite presentation.
\end{Lemma}

\begin{proof}
A reduced presentation $X\to\overline\BTst$ 
composed with $\tau$ is a reduced presentation of 
$\overline{\BTst_n}$. As in the proof of Lemma \ref{Le-Newton}, 
the reduced subscheme $X_\nu$ of $X$ descends to a
reduced substack of $\overline{\BTst_n}$.
Since $\BTst_n$ is of finite type, the
immersion $\BTst_{n,\nu}\to\BTst_n$ is a 
morphisms of finite presentation.
\end{proof}

Along the same lines, one can consider the locus of
infinitesimal groups:

\begin{Lemma}
\label{Le-BT-o}
There are unique reduced closed substacks 
$\BTst^o\subseteq\overline\BTst$ and 
$\BTst^o_n\subseteq\overline{\BTst_n}$ for $n\ge 1$ such 
that the geometric points of $\BTst^o$ and $\BTst^o_n$
are precisely the infinitesimal groups. There is a
Cartesian diagram
$$
\xymatrix@M+0.2em{
\BTst^o \ar[r] \ar[d] &
\BTst^o_{n+1} \ar[r] \ar[d] &
\BTst^o_n \ar[d] \\
\BTst \ar[r]^-{\tau} &
\BTst_{n+1} \ar[r]^-{\tau} & \BTst_n.\!
}
$$
In particular, the closed immersion $\BTst^o\to\BTst$ is
of finite presentation.
\end{Lemma}

\begin{proof}
Let $G$ be a $p$-divisible group or truncated Barsotti-Tate 
group of positive level over an $\FF_p$-scheme $X$.
Since the points of $X$ where the fibre of $G$ is infinitesimal
form a closed subset of $X$, the substacks $\BTst^o$
and $\BTst^o_n$ exist; see the proof of Lemma \ref{Le-Newton}.
The diagram is Cartesian since the truncation morphisms 
$\tau$ are reduced, and since $G$ is infinitesimal if and only 
if $G[p]$ is infinitesimal. The vertical immersions are of
finite presentation because $\BTst_n$ is of finite type.
\end{proof}


\section{The display functor}

\subsection{Frame formalism}

We recall some constructions from \cite{Lau-Frames}
and \cite{Lau-Relation}.
Let $\FFF=(S,I,R,\sigma,\sigma_1)$ be a frame in the
sense of \cite{Lau-Frames} with $p\sigma_1=\sigma$ on $I$. 

In this article, the main example is the following. 
For a $p$-adic ring $R$,
we denote by $W(R)$ the ring of $p$-typical Witt vectors
and by $f$ and $v$ the Frobenius and Verschiebung
of $W(R)$. Let $I_R=v(W(R))$ and let $f_1:I_R\to W(R)$
be the inverse of $v$. Then
$$
\WWW_R=(W(R),I_R,R,f,f_1)
$$ 
is a frame with $pf_1=f$. Windows over $\WWW_R$ in the
sense of \cite{Lau-Frames} are (not necessarily nilpotent)
displays over $R$ in the sense of \cite{Zink-Disp} and
\cite{Messing-Disp}.

For an $S$-module $M$ let 
$M^{(1)}$ be its $\sigma$-twist, and for a $\sigma$-linear
map of $S$-modules $\alpha:M\to N$ let 
$\alpha^\sharp:M^{(1)}\to N$ be its linearisation.
A filtered $F$-$V$-module over $\FFF$ is a
quadruple $(P,Q,F^\sharp,V^\sharp)$ where
$P$ is a projective $S$-module of finite
type with a filtration $IP\subseteq Q\subseteq P$
such that $P/Q$ is projective over $R$, and where
$F^\sharp:P^{(1)}\to P$ and $V^\sharp:P\to P^{(1)}$
are homomorphisms of $S$-modules with
$F^\sharp V^\sharp=p$ and $V^\sharp F^\sharp=p$.
There is a functor
\begin{align*}
\Upsilon:(\FFF\text{-windows})&\to
(\text{filtered $F$-$V$-modules over }\FFF) \\
(P,Q,F,F_1)&\mapsto(P,Q,F^\sharp,V^\sharp)
\end{align*}
such that $F^\sharp$ is the linearisation of $F$, and
$V^\sharp$ is determined by the relation
$V^\sharp(F_1(x))=1\otimes x$ for $x\in Q$;
see \cite[Lemma 10]{Zink-Disp} in the case of displays, 
and \cite[Lemma 2.3]{Lau-Relation}. 
If $S$ has no $p$-torsion, $\Upsilon$ is fully faithful.

Assume that $S$ and $R$ are $p$-adic rings and that $I$
is equipped with divided powers which are compatible
with the canonical divided powers of $p$.
For a $p$-divisible group $G$ over $R$ we denote by 
$\DD(G)$ the \emph{covariant} Dieudonn\'e crystal of $G$. 
By a standard construction, it gives rise to a functor 
\begin{align*}
\Theta:(\pdivC/R)&\to
(\text{filtered $F$-$V$-modules over }\FFF) \\
G&\mapsto(P,Q,F^\sharp,V^\sharp);
\end{align*}
see \cite[Constr.~3.14]{Lau-Relation}.
Here $P=\DD(G)_{S\to R}$,
the submodule $Q$ is the kernel of $P\to\Lie(G)$,
and the Frobenius and Verschiebung of $G\otimes_RR/pR$ 
induce $V^\sharp$ and $F^\sharp$. Note that
$F^\sharp$ is equivalent to a $\sigma$-linear
map $F:P\to P$.

If $S$ has no $p$-torsion, there is a unique 
$\sigma$-linear map $F_1:Q\to P$ such that $(P,Q,F,F_1)$
is an $\FFF$-window that gives back $(P,Q,F^\sharp,V^\sharp)$
when $\Upsilon$ is applied; see \cite[Lemma A.2]{Kisin-Crys} 
and \cite[Prop.~3.15]{Lau-Relation}.
In other words, there is a unique functor
$$
\Phi:(\pdivC/R)\to(\FFF\text{-windows})
$$
together with an isomorphism $\Theta\cong\Upsilon\circ\Phi$.

\subsection{The display functor}

Let $R$ be a $p$-adic ring.
The ideal $I_R$ carries natural divided powers
which are compatible with the canonical divided powers of $p$.
Moreover the ring $W(R)$ is $p$-adic; see \cite[Prop.~3]{Zink-Disp}.
Thus we have a functor $\Theta$ for the frame $\WWW_R$,
which we denote by
$$
\Theta_R:(\pdivC/R)\to
(\text{filtered $F$-$V$-modules over }\WWW_R).
$$
We also have a functor $\Upsilon_{\!R}:(\dispC/R)\to$
(filtered $F$-$V$-modules over $\WWW_R$).

\begin{Prop}
\label{Pr-disp-functor}
For each $p$-adic ring $R$ there is a functor
$$
\Phi_R:(\pdivC/R)
\to(\dispC/R)
$$
together with an isomorphism $\Theta_R\cong\Upsilon_{\!R}\circ\Phi_R$
compatible with base change in $R$. This determines
$\Phi_R$ up to unique isomorphism. 
\end{Prop}

In other words, for each $p$-divisible group $G$ over
a $p$-adic ring $R$ with $\Theta_R(G)=(P,Q,F^\sharp,V^\sharp)$
there is a unique map $F_1:Q\to P$ which is functorial in $G$ and
$R$ such that $(P,Q,F,F_1)$ is a display which induces $V^\sharp$; 
here $F$ is defined by $F(x)=F^\sharp(1\otimes x)$.

\begin{proof}[Proof of Proposition \ref{Pr-disp-functor}]
Let $X=\Spec A\xrightarrow\pi\BTst\times\Spec\ZZ_p$ 
be a reduced presentation, given by a $p$-divisible
group $G$ over $A$; see Lemma \ref{Le-BT-pres}. 
We write $X\times_{\BTst\times\Spec\ZZ_p}X=\Spec B$.
The rings $A$ and $B$ have no $p$-torsion; the 
rings $A/pA$ and $B/pB$ are reduced. Thus the $p$-adic
completions $\hat A$ and $\hat B$ have no $p$-torsion
and are reduced. In particular, the functors 
$\Phi_{\!\hat A}$ and $\Phi_{\hat B}$ exist and 
are unique, which implies that they commute with base
change by arbitrary homomorphisms between $\hat A$ and $\hat B$.

Since displays over a $p$-adic ring $R$ are equivalent
to compatible systems of displays over $R/p^nR$ for $n\ge 1$, 
to prove the proposition it suffices to show that there is 
a unique functor $\Phi_R$ if $p$ is nilpotent in $R$.
Let $H$ be a $p$-divisible group over $R$.
It defines a morphism $\alpha:\Spec R\to\BTst\times\Spec\ZZ/p^m\ZZ$
for some $m$.
We define $S$ and $T$ such that the following diagram
has Cartesian squares, where $\pi_1$ and $\pi_2$ are the 
natural projections.
$$
\xymatrix@!C@M+0.2em{
\Spec T \ar@<0.7ex>[r]^{\psi_1} \ar@<-0.7ex>[r]_{\psi_2} 
\ar[d]^{\alpha''} & 
\Spec S \ar[r]^-\psi \ar[d]^{\alpha'} & 
\Spec R \ar[d]^{\alpha} \\
\Spec B \ar@<0.7ex>[r]^{\pi_1} \ar@<-0.7ex>[r]_{\pi_2} & 
\Spec A \ar[r]^-\pi & \BTst \times\Spec\ZZ_p
}
$$
Then $T\cong S\otimes_RS$ such that $\psi_1$ and $\psi_2$
are the projections, and $\psi$ is faithfully flat. 
Let $H_S=\psi^*H$. We have descent data
$u:\pi_1^*G\cong\pi_2^*G$ and $v:\psi_1^*H_S\cong\psi_2^*H_S$
and an isomorphism $w:\alpha^{\prime*}G\cong H_S$ which preserves
the descent data. Since $p$ is nilpotent in $R$, the pair 
$(\alpha',\alpha'')$ factors into
$$
\xymatrix@C+1em@M+0.2em{
\Spec T \ar@<0.7ex>[d]^{\psi_1} \ar@<-0.7ex>[d]_{\psi_2} 
\ar[r]^{\hat\alpha''} &
\Spec\hat B \ar@<0.7ex>[d]^{\hat\pi_1} \ar@<-0.7ex>[d]_{\hat\pi_2}
\ar[r] &
\Spec B \ar@<0.7ex>[d]^{\pi_1} \ar@<-0.7ex>[d]_{\pi_2} \\ 
\Spec S \ar[r]^{\hat\alpha'} & \Spec\hat A \ar[r] & \Spec A. 
}
$$
The isomorphism $u$ induces
$\hat u:\hat\pi_1^*G_{\!\hat A}\cong\hat\pi_2^*G_{\!\hat A}$,
and $w$ induces an isomorphism 
$\hat w:\hat\alpha^{\prime*}G_{\!\hat A}\cong H_S$
which transforms $\hat u$ into $v$.
The isomorphism $\hat w$
induces an isomorphism of filtered $F$-$V$-modules
$$
\hat\alpha^{\prime*}\Theta_{\!\hat A}(G_{\!\hat A})
\cong\Theta_S(H_S).
$$ 
Thus the operator $F_1$ on $\Theta_{\!\hat A}(G_{\!\hat A})$
given by $\Phi_{\!\hat A}(G_{\!\hat A})$ induces an
operator $F_1$ on $\Theta_S(H_S)$ which makes a display 
$\Phi_S(H_S)$. The descent datum 
$\psi_1^*\Theta_S(H_S)\cong\psi_1^*\Theta_S(H_S)$
induced by $v$ preserves $F_1$ since
$\hat w$ transforms $\hat u$ into $v$ and since
the isomorphism $\pi_1^*\Theta_{\!\hat A}(G_{\!\hat A})\cong 
\pi_1^*\Theta_{\!\hat A}(G_{\!\hat A})$ induced by $\hat u$
preserves $F_1$ by the uniqueness of $\Phi_{\hat B}$.
By fpqc descent, cf.\ \cite[Thm.~37]{Zink-Disp}, 
the operator $F_1$ on $\Theta_S(H_S)$ descends to an
operator $F_1$ on $\Theta_R(H)$ which makes a display $\Phi_R(H)$. 
This display is uniquely determined by the requirement
that the functors $\Phi_{\hat B},\Phi_S,\Phi_R$ are
compatible with base change by the given ring homomorphisms 
$\hat A\to S\leftarrow R$.

The construction implies that $F_1$ is preserved under 
base change by homomorphisms $R\to R'$ of $p$-adic rings
and under isomorphisms
of $p$-divisible groups over $R$. Since a homomorphism
of $p$-divisible groups $g:G\to G'$ can be encoded by
the automorphism $\left(\begin{smallmatrix}1&g\\0&1
\end{smallmatrix}\right)$ of $G'\oplus G$, it follows
that $F_1$ is also preserved under homomorphisms of
$p$-divisible groups over $R$.
\end{proof}

\begin{Prop}
A $p$-divisible group $G$ over a $p$-adic ring
$R$ is infinitesimal (unipotent) if and only 
if the display $\Phi_R(G)$ is nilpotent ($F$-nilpotent).
\end{Prop}

\begin{proof}
The $p$-divisible group $G$ is infinitesimal or unipotent
if and only if the geometric fibres of $G$ in points of 
characteristic $p$ have this property; 
see \cite[Chap.~II, Prop.~4.4]{Messing-Crys}.
Similarly, a display $\PPP=(P,Q,F,F_1)$ over $R$ is nilpotent or 
$F$-nilpotent if and only if the geometric fibres of $\PPP$ 
in points of characteristic $p$ have this property. 
Indeed, let $\bar P$ be the projective $R/pR$-module $P/(I_RP+pP)$. 
The display $\PPP$ is nilpotent (resp.\ $F$-nilpotent)
if and only if the homomorphism 
$\bar V^\sharp:\bar P\to\bar P^{(1)}$ 
(resp.\ $\bar F^\sharp:\bar P^{(1)}\to\bar P$) is nilpotent,
which can be verified at the geometric points since $\bar P$
is finitely generated.
Thus the proposition follows from the case of
perfect fields, which is well-known.
\end{proof}

\begin{Remark}
\label{Re-Phi-dual}
There is a natural duality isomorphism 
$\Phi_R(G^\vee)\cong\Phi_R(G)^t$, where $G^\vee$ is the Serre 
dual of $G$, and where ${}^t$ denotes the dual display as
in \cite[Def.~19]{Zink-Disp}. 
Indeed, the crystalline duality theorem \cite[5.3]{BBM} 
implies that the functor $\Theta_R$ is compatible with
duality, and the assertion follows from the uniqueness
part of Proposition \ref{Pr-disp-functor}.
See also \cite[Cor.~3.26]{Lau-Relation}.
\end{Remark}

\subsection{The extended display functor}

Assume that $B\to R$ is a surjective homomorphism 
of $p$-adic rings whose kernel $\Fb\subset B$ is equipped
with divided powers $\delta$ that are compatible with
the canonical divided powers of $p$. Then one can define
a frame 
$$
\WWW_{B/R}=(W(B),I_{B/R},R,f,\tilde f_1)
$$
with $p\tilde f_1=f$;
see \cite[Section 2.2]{Lau-Relation}.
Windows over $\WWW_{B/R}$ are called displays for $B/R$.
The ideal $I_{B/R}$ carries natural divided powers,
depending on $\delta$, which are
compatible with the canonical divided powers of $p$;
see \cite[Section 2.7]{Lau-Frames}. 
Thus there is a functor $\Theta$ for $\WWW_{B/R}$,
which we denote by
$$
\Theta_{B/R}:(\pdivC/R)\to
(\text{filtered $F$-$V$-modules over }\WWW_{B/R}).
$$
We also have $\Upsilon_{\!B/R}:$ (displays for
$B/R$) $\to$ (filtered $F$-$V$-mod.\ over $\WWW_{B/R}$).

\begin{Prop}
\label{Pr-disp-functor-pd}
For each divided power extension $B\to R$ of $p$-adic rings
which is compatible with the canonical divided powers of $p$,
there is a functor
$$
\Phi_{B/R}:(\pdivC/R)
\to(\text{displays for }B/R)
$$
together with an isomorphism 
$\Theta_{B/R}\cong\Upsilon_{\!B/R}\circ\Phi_{B/R}$
compatible with base change in $(B,R,\delta)$.
This determines $\Phi_{B/R}$ up to unique isomorphism.
\end{Prop}

\begin{proof}[Proof of Proposition \ref{Pr-disp-functor-pd}]
We may assume that $p$ is nilpotent in $B$.
For a given $p$-divisible group $H$ over $R$ we
choose a $p$-divisible group $H_1$ over $B$
which lifts $H$; this is possible since $B\to R$ 
is a nil-extension due to the divided powers.
Necessarily we have to define $\Phi_{B/R}(H)$
as the base change of $\Phi_{B}(H_1)$ under the
natural frame homomorphism $\WWW_{B}\to\WWW_{B/R}$. Here 
$\Phi_B$ is well-defined by Proposition \ref{Pr-disp-functor}.
We have to show that the operator $F_1$ defined in this
way on $\Theta_{B/R}(H)$ does not depend on the choice
of $H_1$; then it follows easily that $F_1$ is compatible
with base change in $B/R$ and commutes with homomorphisms
of $p$-divisible groups over $R$. 

As in the proof of Proposition \ref{Pr-disp-functor},
the assertion is reduced to a universal situation.
Let $X=\Spec A\xrightarrow\pi\BTst\times\Spec\ZZ_p$ be a 
presentation given by a $p$-divisible group $G$ over $X$
such that $\pi$ and $X\to\Spec\ZZ_p$ are $\infty$-smooth 
coverings; see Lemma \ref{Le-BT-pres}.
Let $X'=\Spec A'$ be the $p$-adic completion of the divided
power envelope of the diagonal $X\to X\times_{\Spec\ZZ_p}X$
and let $G_1,G_2$ be the inverse images of $G$ under the
two projections $X'\to X$. These are two lifts of $G$ with
respect to the diagonal morphism $X\to X'$.
Since the divided power envelope
of the diagonal of a smooth $\ZZ_p$-algebra has no
$p$-torsion and since $A$ is the direct limit of smooth 
$\ZZ_p$-algebras, $A'$ has no $p$-torsion, and thus $W(A')$
has no $p$-torsion. Thus the operators $F_1$ on 
$\Theta_{A'/A}(G)$ defined by $\Phi_{A'}(G_1)$ and by
$\Phi_{A'}(G_2)$ are equal.

The given $p$-divisible group $H$ over $R$ defines 
a morphism $\alpha:\Spec R\to\BTst\times\Spec\ZZ_p$.
Since $\pi$ is an $\infty$-smooth covering and since
a surjective smooth morphism of schemes has a section
\'etale locally in the base, we can find a ring $R'$ and
a commutative diagram
$$
\xymatrix@M+0.2em{
\Spec R' \ar[d]_\psi \ar[r]^-{\alpha'} & X \ar[d]^\pi \\
\Spec R \ar[r]^-\alpha & \BTst\times\Spec\ZZ_p
}
$$ 
where $\psi$ ind-\'etale and surjective.
Since $\Spec R\to\Spec B$ is a nil-immersion, there
is a unique ind-\'etale and surjective morphism
$\Spec B'\to\Spec B$ which extends $\psi$.
Since $B'$ is flat over $B$, the given divided powers on 
the kernel of $B\to R$ extend to divided powers on the
kernel of $B'\to R'$.
Let $H_1$ and $H_2$ be two lifts of $H$ to $B$ and let 
$\beta_i:\Spec B'\to\BTst\times\Spec\ZZ_p$ be the morphism
given by $H_i\otimes_BB'$ for $i=1,2$. We have a 
commutative diagram
$$
\xymatrix@M+0.2em{
\Spec R'\ar[r]^-{\alpha'} \ar[d] & X \ar[d]^\pi \\
\Spec B' \ar[r]^-{\beta_i} & \BTst\times\Spec\ZZ_p.\!
}
$$
Since $\pi$ is an $\infty$-smooth covering and since
a smooth morphism 
satisfies the lifting criterion of formal smoothness for
arbitrary nil-immersions of affine schemes, there are morphisms 
$\beta_i':\Spec B'\to X$ for $i=1,2$ such that in the
preceding diagram both triangles commute. They define 
a morphism $\beta':\Spec B'\to X\times_{\Spec\ZZ_p}X$,
which factors uniquely over $\beta'':\Spec B'\to X'$, and
we have isomorphisms $H_i\otimes_BB'\cong\beta^{\prime\prime*}G_i$
that lift the given isomorphism $H\otimes_RR'\cong\alpha^{\prime*}G$.
Thus the operators $F_1$ on $\Theta_{B'/R'}(H\otimes_RR')$ defined 
by $\Theta_{B'}(H_1\otimes_BB')$ and by $\Theta_{B'}(H_2\otimes_BB')$ 
are equal. Since $W(B)\to W(B')$ is injective, it follows that
the operators $F_1$ on $\Theta_{B/R}(H)$ defined by
$\Theta_{B}(H_1)$ and by $\Theta_{B}(H_2)$ are equal as well.
\end{proof}

\begin{Remark}
As in Remark \ref{Re-Phi-dual} we have 
$\Phi_{B/R}(G^\vee)\cong\Phi_{B/R}(G)^t$.
\end{Remark}

\subsection{Consequences}

Let us recall the Dieudonn\'e crystal of a nilpotent display.
If $B\to R$ is a divided power extension of $p$-adic rings 
which can be written as the projective limit of divided power
extensions $B_n\to R_n$ with $p^nB_n=0$, for example
if the divided powers are compatible with the
canonical divided powers of $p$, the base change
functor from nilpotent displays for $B/R$ to nilpotent
displays over $R$ is an equivalence of categories by
\cite[Thm.~44]{Zink-Disp}. Using this, one defines the
Dieudonn\'e crystal $\DD(\PPP)$ of a nilpotent display 
$\PPP$ over $R$ as follows: If $(\tilde P,\tilde Q,F,F_1)$ is
the unique lift of $\PPP$ to a display for $B/R$,
$$
\DD(\PPP)_{B/R}=\tilde P/I_B\tilde P.
$$
By duality, the same applies to $F$-nilpotent displays.

\begin{Remark}
\label{Re-Groth-Messing}
It is well-known that the crystalline deformation theorem
for $p$-divisible groups in \cite{Messing-Crys} holds for not
necessarily nilpotent divided powers if the
groups are infinitesimal or unipotent; see
\cite[Sec.~4]{Grothendieck-Nice} and 
\cite[p.~111]{Grothendieck-Montreal}.

More precisely, assume that $B\to R=B/\Fb$ is a divided 
power extension of rings in which $p$ is nilpotent. 
We have a natural functor $G\mapsto(H,V)$ from $p$-divisible
groups $G$ over $B$ to $p$-divisible groups $H$ over $R$ 
together with a lift of the Hodge filtration of $H$ to a 
direct summand $V\subset\DD(H)_{B/R}$. Here 
$\DD(H)$ is the covariant Dieudonn\'e crystal given by \cite{BBM}.
If the divided powers on $\Fb$ are nilpotent, this functor 
is an equivalence of categories by the deformation theorem
\cite[V.1.6]{Messing-Crys} together with the comparison
of the crystals of \cite{Messing-Crys} and of \cite{BBM} 
in \cite[3.2.11]{Berthelot-Messing}.
For general divided powers, the functor induces an equivalence between
unipotent (infinitesimal) $p$-divisible groups $G$ over $B$
and pairs $(H,V)$ where $H$ is unipotent (infinitesimal).

Let us indicate a proof of this fact. First,
by the crystalline duality theorem \cite[5.3]{BBM}
it suffices to consider unipotent $p$-divisible groups.

For each commutative formal Lie group $L=\Spf A$ over $B$ 
there is a homomorphism 
$\log_L:\Ker[L(B)\to L(R)]\to\Ker[\Lie(L)\to\Lie(L_R)]$
such that $\log_{\hat\GG_m}$ 
is given by the usual logarithm series.
It can be described as follows:
The Cartier dual $L^*=\uHom(L,\Gm)$ is an
affine group scheme $\Spec A^*$. We have isomorphisms 
$L\cong\uHom(G^*,\Gm)$ and $\uLie(L)\cong\uHom(G^*,\Ga)$, 
under which $\log_L$ is induced by $\log_{\hat\GG_m}$.
If the divided powers on $\Fb$ are nilpotent,
$\log_G$ is an isomorphism; its inverse is given by
the usual exponential series for $\Gm$. 
Let us call $L$ unipotent if the augmentation ideal 
$J\subset A^*$ is a nilideal. If $L$ is unipotent,
arbitrary divided powers on $\Fb$ induce nilpotent
divided powers on $\Fb J$, which implies that $\log_L$ 
is an isomorphism again.

The construction of the Dieudonn\'e crystal 
for nilpotent divided powers in \cite{Messing-Crys}
is based on the exponential of the formal completion of 
the universal vector extension $EG$ of a $p$-divisible 
group $G$ over $B$.
If $G$ is a unipotent $p$-divisible
group, the formal completion of $EG$ is a unipotent 
formal Lie group. Therefore, in the case of unipotent 
$p$-divisible groups, the construction of the Dieudonn\'e
crystal of \cite{Messing-Crys} and the proof of the
deformation theorem \cite[V.1.6]{Messing-Crys} are
valid for not necessarily nilpotent divided powers.
The comparison of crystals in \cite[3.2.11]{Berthelot-Messing} 
carries over to this case as well.
\qed
\end{Remark}

Recall that we denote by 
$\DD(G)$ the covariant Dieudonn\'e crystal
of a $p$-divisible group $G$.

\begin{Cor}
\label{Co-DD(G)-DD(PhiG)}
Let $B\to R$ be a divided power extension of $p$-adic rings 
which is compatible with the canonical divided powers of $p$, 
and let $G$ be a $p$-divisible group over $R$ which is unipotent 
or infinitesimal. There is a natural isomorphism of projective
$B$-modules 
$$
\DD(G)_{B/R}\cong\DD(\Phi_R(G))_{B/R}.
$$
\end{Cor}

\begin{proof}
If $G$ is infinitesimal (unipotent) then $\Phi_R(G)$
is nilpotent ($F$-nilpotent), and
the display $\Phi_{B/R}(G)=(\tilde P,\tilde Q,F,F_1)$ is the 
unique lift of $\Phi_R(G)$ to a nilpotent ($F$-nilpotent)
display for $B/R$. Since $\tilde P=\DD(G)_{W(B)\to R}$ 
and since the projection $W(B)\to B$ is a homomorphism of 
divided power extensions of $R$, we get an isomorphism
$\DD(G)_{B/R}\cong\tilde P/I_B\tilde P=\DD(\Phi_R(G))_{B/R}$. 
\end{proof}

\begin{Cor}
\label{Co-Phi-form-etale}
The restriction of\/ $\Phi_R$ to infinitesimal or unipotent
$p$ di\-vis\-ible groups is formally \'etale, i.e.\ for a
surjective homomorphism $B\to R$ of rings in which $p$
is nilpotent with nilpotent kernel, the category of
infinitesimal (unipotent) $p$-divisible groups over
$B$ is equivalent to the category of such groups $G$
over $R$ together with a lift of\/ $\Phi_R(G)$ to a
display over $B$.
\end{Cor}

\begin{proof}
If the kernel of $B\to R$ carries divided powers compatible
with the canonical divided powers of $p$, in view of
Corollary \ref{Co-DD(G)-DD(PhiG)} the assertion follows
from the crystalline deformation theorem of $p$-divisible groups 
(see Remark \ref{Re-Groth-Messing}) and its counterpart
for nilpotent (unipotent) displays \cite[Prop.~45]{Zink-Disp}.
Thus the corollary holds for $B\to B/pB$ and for $R\to R/pR$.
Hence we may assume that $B$ is annihilated by $p$.
Then $B\to R$ is a finite composition of divided power
extensions with trivial divided powers, which are
compatible with the canonical divided powers of $p$
since $pB$ is zero.
\end{proof}

The following is a special case of Theorem \ref{Th-formal-classif}.

\begin{Cor}
Let $R$ be a complete local ring with perfect residue field
of characteristic $p$. The functor $\Phi_R$ induces an
equivalence between $p$-divisible groups over $R$ with
infinitesimal (unipotent) special fibre and displays over
$R$ with nilpotent ($F$-nilpotent) special fibre.
\end{Cor}

\begin{proof}
Over perfect fields this is classical. The general
case follows by Corollary \ref{Co-Phi-form-etale}
and by passing to the limit over $R/\Fm_R^n$.
\end{proof}


\section{Truncated displays}

\subsection{Preliminaries}

For a $p$-adic ring $R$ and a positive integer $n$ let 
$W_n(R)$ be the ring of truncated Witt vectors of length $n$ 
and let $I_{n,R}\subset W_n(R)$ be the kernel of
the augmentation to $R$. The Frobenius of $W(R)$
induces a ring homomorphism 
$$
f:W_{n+1}(R)\to W_n(R).
$$
The inverse of the Verschiebung of $W(R)$ induces a 
bijective $f$-linear map 
$$
f_1:I_{n+1,R}\to W_n(R).
$$
If $R$ is an $\FF_p$-algebra, the
Frobenius of $W(R)$ induces a ring endomorphism $f$ of 
$W_n(R)$, and the ideal $I_{n+1}(R)$ of $W_{n+1}(R)$ 
is a $W_n(R)$-module. 

\begin{Defn}
\label{Def-pre-disp}
A \emph{pre-display} over an $\FF_p$-algebra $R$ is a sextuple
$$
\PPP=(P,Q,\iota,\varepsilon,F,F_1)
$$ 
where $P$ and $Q$ are $W(R)$-modules with homomorphisms
$$
\xymatrix@M+0.2em{
I_{R}\otimes_{W(R)}P \ar[r]^-\varepsilon & Q \ar[r]^\iota & P
}
$$
and where $F:P\to P$ and $F_1:Q\to P$ are $f$-linear maps
such that the following relations hold: 
The compositions $\iota\varepsilon$
and $\varepsilon(1\otimes\iota)$ are the multiplication
homomorphisms, and we have $F_1\varepsilon=f_1\otimes F$.

If $P$ and $Q$ are $W_n(R)$-modules, 
we call $\PPP$ a pre-display of level $n$.
\end{Defn}

The axioms imply that $F\iota=pF_1$; 
cf.\ \cite[Eq.~(2)]{Zink-Disp}.

\medskip

Pre-displays over $R$ form an abelian category
$(\predispC/R)$ which contains the category of
displays $(\dispC/R)$ as a full subcategory. 
Pre-displays of level $n$ over $R$ form an abelian subcategory
$(\predispC_n/R)$ of $(\predispC/R)$.
For a homomorphism of $p$-adic rings $\alpha:R\to R'$,
the restriction of scalars defines a functor 
$$
\alpha^*:(\predispC/R')\to(\predispC/R).
$$
It has a left adjoint $\PPP\mapsto W(R')\otimes_{W(R)}\PPP$ 
given by the tensor product in each component.
The restriction of $\alpha^*$ to pre-displays of level $n$
is a functor
$
(\predispC_n/R')\to(\predispC_n/R)
$
with left adjoint $\PPP\mapsto W_n(R')\otimes_{W_n(R)}\PPP$.

\subsection{Truncated displays}

Truncated displays of level $n$ are pre-displays of level
$n$ with additional properties. We begin with the conditions
imposed on the linear data of the pre-display.
For an $\FF_p$-algebra $R$ let
$$
J_{n+1,R}=\Ker(W_{n+1}(R)\to W_n(R)).
$$ 

\begin{Defn}
\label{Def-trunc-pair}
A \emph{truncated pair} of level $n$ over an 
$\FF_p$-algebra $R$ is a 
quadruple $\BBB=(P,Q,\iota,\varepsilon)$ where
$P$ and $Q$ are $W_n(R)$-modules with homomorphisms
$$
\xymatrix@M+0.2em{
I_{n+1,R}\otimes_{W_n(R)}P \ar[r]^-\varepsilon & Q \ar[r]^\iota & P
}
$$
such that the following properties hold.
\begin{enumerate}
\item
The compositions $\iota\varepsilon$ and $\varepsilon(1\otimes\iota)$
are the multiplication maps.
\item
The $W_n(R)$-module $P$ is projective of finite type.
\item
The $R$-module $\Coker(\iota)$ is projective.
\item
We have an exact sequence, where 
$\bar\varepsilon$ is induced by $\varepsilon$:
\begin{equation}
\label{Eq-trunc-4term}
0\to J_{n+1,R}\otimes_R\Coker(\iota)\xrightarrow{\bar\varepsilon}
Q\xrightarrow\iota P\to\Coker(\iota)\to 0.
\end{equation}
\end{enumerate}
A \emph{normal decomposition} for a truncated pair consists
of projective $W_n(R)$-modules $L\subseteq Q$ and
$T\subseteq P$ such that we have bijective homomorphisms
$$
L\oplus T\xrightarrow{\iota+1}P,\qquad
L\oplus(I_{n+1,R}\otimes_{W_n(R)}T)\xrightarrow{1+\varepsilon}Q.
$$
\end{Defn}

Each pair $(L,T)$ of projective $W_n(R)$-modules of finite
type defines a unique truncated pair for which $(L,T)$ is
a normal decomposition.

\begin{Lemma}
\label{Le-pair-norm-dec}
Every truncated pair $\BBB$ admits a normal decomposition.
\end{Lemma}

\begin{proof}
Let $\bar L=\Coker(\varepsilon)$, an $R$-module,
and $\bar T=\Coker(\iota)$, a projective $R$-module.
The 4-term exact sequence 
\eqref{Eq-trunc-4term} induces a short exact sequence
\begin{equation}
\label{Eq-trunc-3term}
0\to\bar L\xrightarrow{\bar\iota} P/I_{n,R}P\to\bar T\to 0.
\end{equation}
Thus $\bar L$ is a projective $R$-module. Let $L$ and
$T$ be projective $W_n(R)$-modules which lift
$\bar L$ and $\bar T$. Let $L\to Q$ and $T\to P$
be homomorphisms which commute with the obvious
projections to $\bar L$ and $\bar T$, respectively.
The exact sequence \eqref{Eq-trunc-3term} implies that
the homomorphism $\iota+1:L\oplus T\to P$ becomes an
isomorphism over $R$, so it is an isomorphism as both sides
are projective. Let $\BBB'$ be the truncated pair defined
by $(L,T)$. We have a homomorphism of truncated pairs
$\BBB'\to\BBB$ such that the associated homomorphism
of the 4-term sequences \eqref{Eq-trunc-4term} is an
isomorphism except possibly at $Q$.
Hence it is an isomorphism by the 5-Lemma.
\end{proof}

\begin{Defn}
\label{Def-trunc-disp}
A truncated display of level $n$ over an $\FF_p$-algebra
$R$ is pre-display $\PPP=(P,Q,\iota,\varepsilon,F,F_1)$
over $R$ such $(P,Q,\iota,\varepsilon)$ is 
a truncated pair of level $n$ and such that the image of $F_1$ 
generates $P$ as a $W_n(R)$-module. 
\end{Defn}

Let $(\dispC_n/R)$ be the category of truncated displays 
of level $n$ over $R$. This is a full subcategory of the 
abelian category $(\predispC_n/R)$.

If $(P,Q,\iota,\varepsilon)$ is a truncated pair with
given normal decomposition $(L,T)$, the set of pairs $(F,F_1)$
such that $(P,Q,\iota,\varepsilon,F,F_1)$ is a truncated
display is bijective to the set of $f$-linear isomorphisms
$\Psi:L\oplus T\to P$ such that $\Psi|_L=F_1|_L$
and $\Psi|_T=F|_T$. If $L$ and $T$ are free $W_n(R)$-modules,
$\Psi$ is described by an invertible matrix over $W_n(R)$.
This is analogous to the case of displays; 
see \cite[Lemma 9]{Zink-Disp} and the subsequent discussion.
The triple $(L,T,\Psi)$ is called a normal representation 
of $(P,Q,\iota,\varepsilon,F,F_1)$.

Let $k$ be a perfect field of characteristic $p$.
A truncated Dieudonn\'e module of level $n$ over $k$
is a triple $(M,F,V)$ where $M$ is a free $W(k)$-module 
of finite rank with an $f$-linear map $F:M\to M$ and an 
$f^{-1}$-linear map $V:M\to M$ such that $FV=p=VF$.
If $n=1$ we require that $\Ker F=\Image V$, which is
equivalent to $\Ker V=\Image F$.

\begin{Lemma}
\label{Le-Dieud-disp-field}
Truncated displays of level $n$ over a perfect field
$k$ are equivalent
to truncated Dieudonn\'e modules of level $n$ over $k$.
\end{Lemma}

\begin{proof}
Multiplication by $p$ gives an isomorphism 
$W_n(k)\cong I_{n+1,k}$. Thus truncated displays
of level $n$ are equivalent to quintuples
$\PPP=(P,Q,\iota,\epsilon,F_1)$ where $P$ and $Q$
are free $W_n(k)$-modules with homomorphisms
$P\xrightarrow{\epsilon}Q\xrightarrow\iota P$
with $\epsilon\iota=p$ and $\iota\epsilon=p$
such that the sequence
$$
Q\xrightarrow\iota
P\xrightarrow{p^{n-1}\epsilon}Q\xrightarrow\iota P
$$
is exact, and where $F_1:Q\to P$ is a bijective $f$-linear map. 
The exactness is automatic if $n\ge 2$.
The operator $F$ of the truncated display
is given by $F=F_1\epsilon$. Let
$V=\iota F_1^{-1}$. The assignment $\PPP\mapsto(P,F,V)$
is an equivalence between truncated displays and
truncated Dieudonn\'e modules.
\end{proof}

\begin{Lemma}
\label{Le-trunc-base-change}
For a homomorphism of\/ $\FF_p$-algebras $\alpha:R\to R'$ 
there is a unique base change functor
$$
\alpha_*:(\dispC_n/R)\to(\dispC_n/R')
$$
with a natural isomorphism 
$$
\Hom_{(\predispC_n/R)}(\PPP,\alpha^*\PPP')\cong
\Hom_{(\dispC_n/R)}(\alpha_*\PPP,\PPP')
$$
for all truncated displays $\PPP$ of level $n$ 
over $R$ and $\PPP'$ of level $n$ over $R'$.
\end{Lemma}

\begin{proof}
This is straightforward. In terms 
of normal representations, $\alpha_*$ is given by
$(L,T,\Psi)\mapsto(W_n(R')\otimes_{W_n(R)}L,
W_n(R')\otimes_{W_n(R)}T,f\otimes\Psi)$.
\end{proof}

\begin{Remark}
\label{Re-ind-et}
If $\alpha$ is ind-\'etale, then $W_n(R)\to W_n(R')$
is ind-\'etale, and the ideal $v^m(I_{{n-m},R'})$ is equal to
$W_n(R')\otimes_{W_n(R)}v^m(I_{n-m,R})$ for $0\le m\le n$. 
This is proved in 
\cite[Prop.~A.8]{Langer-Zink-dRW} if $\alpha$ is \'etale, and the 
functor $W_n$ preserves filtered direct limits of rings. 
As a consequence we obtain:
\end{Remark}

\begin{Cor}
\label{Co-trunc-base-change-ind-et}
For each truncated display $\PPP$ of level $n$ over $R$
there is a natural homomorphism of pre-displays over $R'$
$$
W_n(R')\otimes_{W_n(R)}\PPP\to\alpha_*\PPP.
$$
If $\alpha$ is ind-\'etale, this homomorphism is an isomorphism.
\end{Cor}

\begin{proof}
In view of Remark \ref{Re-ind-et}, this follows from the 
proof of Lemma \ref{Le-trunc-base-change}.
\end{proof}

\begin{Lemma}
\label{Le-trunc-disp-local}
Assume that $\alpha:R\to R'$ 
is a faithfully flat ind-\'etale homomorphism
of $\FF_p$-algebras. If $\PPP$ is a pre-display of level
$n$ over $R$ such that $\PPP'=W_n(R')\otimes_{W_n(R)}\PPP$ is
a truncated display of level $n$ over $R'$, then $\PPP$ 
is a truncated display of level $n$ over $R$.
\end{Lemma}

\begin{proof}
The pre-display $\PPP$ is a truncated display if and only if
$P$ is projective of finite type over $W_n(R)$, the homomorphism
$F_1^\sharp:Q^{(1)}\to P$ is surjective, and the
4-term sequence \eqref{Eq-trunc-4term} is exact.
These properties descend from $\PPP'$ to $\PPP$
since in all components of $\PPP$ and of \eqref{Eq-trunc-4term},
the passage from $\PPP$ to $\PPP'$ is given by the tensor product 
with the faithfully flat homomorphism 
$W_n(\alpha)$; see Remark \ref{Re-ind-et}.
\end{proof}

\begin{Lemma}
For each $\FF_p$-algebra $R$ there are unique
truncation functors
$$
\tau_n:(\dispC/R)\to(\dispC_n/R),
$$
$$
\tau_n:(\dispC_{n+1}/R)\to(\dispC_n/R),
$$
together with a natural isomorphism
$$
\Hom_{(\predispC/R)}(\PPP,\PPP')\cong
\Hom_{(\dispC_n/R)}(\tau_n\PPP,\PPP')
$$
if $\PPP$ is a display or truncated display of level $n+1$ over
$R$ and if $\PPP'$ is a truncated display of level $n$ over $R$.
The truncation functors are compatible with base change in $R$.
\end{Lemma}

\begin{proof}
Again this is straightforward.
In terms of normal representations, $\tau_n$ is given by
$(L,T,\Psi)\mapsto(W_n(R)\otimes_{W(R)}L,
W_n(R)\otimes_{W(R)}T,f\otimes\Psi)$.
\end{proof}

\begin{Lemma}
\label{Le-disp-trunc-disp}
For an $\FF_p$-algebra $R$, the category $(\dispC/R)$ 
is the projective limit over $n$ of the
categories $(\dispC_n/R)$.
\end{Lemma}

\begin{proof}
It is easy to see that the truncation functor from 
displays over $R$ to compatible systems of truncated 
displays of level $n$ over $R$ is fully faithful. 
For a given compatible system of truncated displays 
$(\PPP_n)_{n\ge 1}$ we define $\PPP=\varprojlim_n\PPP_n$ 
componentwise.
The proof of Lemma \ref{Le-pair-norm-dec} shows
that each normal decomposition of $\PPP_n$ can
be lifted to a normal decomposition of $\PPP_{n+1}$.
It follows easily that $\PPP$ is a display.
\end{proof}

\subsection{Descent}

We recall the descent of projective modules over  
truncated Witt rings. 
Let $R\to R'$ be a faithfully flat homomorphism of rings 
in which $p$ is nilpotent and let $R''=R'\otimes_RR'$.
We denote by $\VVV_n(R)$ the category
of projective $W_n(R)$-modules of finite type and by
$\VVV_n(R'/R)$ the category of modules in $\VVV_n(R')$
together with a descent datum relative to $R\to R'$.

\begin{Lemma}
\label{Le-Wn-descent}
The obvious functor $\gamma:\VVV_n(R)\to\VVV_n(R'/R)$ is an
equivalence.
\end{Lemma}

\begin{proof}
First we note that for each flat $W_n(R)$-module $M$ the complex
$$
C_n(M)=[0\to M\to M\otimes_{W_n(R)}W_n(R')\to
M\otimes_{W_n(R)}W_n(R'')\to\cdots]
$$
is exact. Indeed, this is clear if $n=1$, and in general
$C_n(M)$ is an extension of $C_1(M\otimes_{W_n(R),f^{n-1}}R)$
and $C_{n-1}(M\otimes_{W_n(R)}W_{n-1}(R))$.

It follows that the functor $\gamma$ is fully faithful.
We have to show that $\gamma$ is essentially surjective.
If $R'$ is a finite product of localisations of $R$ then
$W_n(R)\to W_n(R')$ has the same property, and thus $\gamma$ 
is an equivalence. Hence we may always pass to an open
cover of $\Spec R$ by spectra of localisations of $R$.
For $M'$ in $\VVV_n(R'/R)$ the descent datum induces a 
descent datum for the projective $R'$-module $M'/I_{R'}M'$, 
which is effective. By passing to a localisation of $R$ 
we may assume that the descended $R$-module is free. 

For fixed $r$ let $\VVV_n^o(R)$ be the category of modules
$M$ in $\VVV_n(R)$ together with an isomorphism of $R$-modules
$\beta:R^r\cong M/I_RM$; homomorphisms in $\VVV_n^o(R)$
preserve the $\beta$'s. In view of the preceding
remarks it suffices to show that the obvious functor
$\VVV_n^o(R)\to\VVV_n^o(R'/R)$ is essentially surjective.
Since every object of $\VVV_n^o(R)$ is isomorphic to the 
standard object $(W_n(R)^r,\id)$, we have to show
that all objects in $\VVV_n^o(R'/R)$ are isomorphic.
This means that for the sheaf of groups $\underline A$ on the 
category of affine $R$-schemes defined by 
$\underline A(\Spec S)=\Aut(W_n(S)^r,\id)$, the \v Cech 
cohomology group $\check H^1(R'/R,\underline A)$ is trivial. 
This is true because $\underline A$ has a finite filtration 
with quotients isomorphic to quasi-coherent modules.
\end{proof}

We turn to descent of truncated pairs. Let $R\to R'$
be a faithfully flat homomorphism of $\FF_p$-algebras.
We denote by $\CCC_n(R)$
the category of truncated pairs of level $n$ over $R$
and by $\CCC_n(R'/R)$ the category of truncated
pairs of level $n$ over $R'$ together with a descent
datum relative to $R\to R'$.

\begin{Lemma}
\label{Le-pair-descent}
The obvious functor $\gamma:\CCC_n(R)\to\CCC_n(R'/R)$ is
an equivalence.
\end{Lemma}

\begin{proof}
For a truncated pair $\BBB$ over $R$ we denote by
$\BBB',\BBB''$ the base change to $R',R''$ etc.
We have an exact sequence $0\to P\to P'\to P''$ by 
the proof of Lemma \ref{Le-Wn-descent}, and thus
$0\to Q\to Q'\to Q''$ by the 4-term exact sequence
\eqref{Eq-trunc-4term}. It follows easily that the
functor $\gamma$ is fully faithful. We show that
$\gamma$ is essentially surjective by a variant of
the proof of Lemma \ref{Le-Wn-descent}.
Again, the assertion holds if $R'$ is a finite product 
of localisations of $R$, and thus we may pass to an
open cover of $\Spec R$ defined by localisations.
For $\BBB'$ in $\CCC_n(R'/R)$ the given descent datum
induces a descent datum for the projective $R$-modules
$\Coker(\iota)$ and $\Coker(\varepsilon)$. By passing
to a localisation of $R$ we may assume that the
descended $R$-modules are free.

For fixed $r,s$ let $\CCC_n^o(R)$ be the category of
truncated pairs $\BBB$ in $\CCC_n(R)$ together with
isomorphisms $\beta_1:R^r\cong\Coker(\iota)$
and $\beta_2:R^s\cong\Coker(\varepsilon)$ of $R$-modules; homomorphisms
in $\CCC_n^o(R)$ preserve the $\beta_i$. It suffices to
show that $\CCC_n^o(R)\to\CCC_n^o(R'/R)$ is essentially
surjective. By Lemma \ref{Le-pair-norm-dec} and its proof,
all objects of $\CCC_n^o(R)$ are isomorphic. For fixed 
$(\BBB,\beta_1,\beta_2)$ in $\CCC_n^o(R)$ with normal
decomposition $(L,T)$ the group $\Aut(\BBB,\beta_1,\beta_2)$ 
can be identified with the group of matrices
$\left(\begin{smallmatrix}A&B\\C&D\end{smallmatrix}\right)$
with $A\in\Aut(L)$, $B\in\Hom(T,L)$, 
$C\in\Hom(L,I_{n+1,R}\otimes_{W_n(R)}T)$, and $T\in\Aut(T)$
such that $A\equiv\id$ and $D\equiv\id$ modulo $I_R$.
The sheaf of groups $\underline A=\uAut(\BBB,\beta_1,\beta_2)$ has
a finite filtration with quotients isomorphic to quasi-coherent
modules. Thus $\check H^1(R'/R,\underline A)$ is trivial,
which implies that all objects of $\CCC_n^o(R'/R)$ are
isomorphic.
\end{proof}

\begin{Prop}
\label{Pr-trunc-descent}
Faithfully flat descent is effective for truncated displays
over $\FF_p$-algebras.
\end{Prop}

\begin{proof}
By Lemmas \ref{Le-pair-descent} and \ref{Le-pair-norm-dec} 
it suffices to show that for a given truncated pair $\BBB$ over 
an $\FF_p$-algebra $R$ with a normal decomposition $(L,T)$, 
the truncated display structures on $\BBB$ form an fpqc sheaf 
on the category of affine schemes over $\Spec R$.  
This is true because these structures correspond to
$f$-linear isomorphisms $L\oplus T\to P$.
\end{proof}

\subsection{Smoothness}

As in section \ref{Subse-stack-pdiv} we fix a positive integer $h$.
We denote by $\Disp_n\to\Spec\FF_p$ the stack of truncated 
displays of level $n$ and rank $h$. Thus $\Disp_n(\Spec R)$ 
is the groupoid of truncated displays 
of level $n$ and rank $h$ over $R$. 
The truncation functors induce morphisms
$$
\tau_n:\Disp_{n+1}\to\Disp_n.
$$

\begin{Prop}
\label{Pr-Disp-n-smooth}
The fibered category $\Disp_n$ is a smooth Artin algebraic
stack of dimension zero over $\FF_p$ with affine diagonal. 
The morphism $\tau_n$ is smooth and surjective 
of relative dimension zero.
\end{Prop}

\begin{proof}
By Proposition \ref{Pr-trunc-descent}, $\Disp_n$ is an
fpqc stack. In order to see that its diagonal is affine 
we have to show that for truncated displays $\PPP_1$ and
$\PPP_2$ over an $\FF_p$-algebra $R$ the sheaf 
$\uIsom(\PPP_1,\PPP_2)$ is represented by an affine
scheme. By passing to an open cover of $\Spec R$ we may
assume that $\PPP_1$ and $\PPP_2$ have normal decompositions
by free modules. Then homomorphisms of the underlying
truncated pairs are represented by an affine space.
To commute with $F$ and $F_1$ is a closed condition,
and a homomorphism of truncated pairs 
is an isomorphism if and only if it induces isomorphisms
on $\Coker(\iota)$ and $\Coker(\varepsilon)$,
which means that two determinants are invertible. Thus
$\uIsom(\PPP_1,\PPP_2)$ is an affine scheme.

For each integer $d$ with $0\le d\le h$ let 
$\Disp_{n,d}$ be the substack of $\Disp_n$ where 
the projective module $\Coker(\iota)$ has rank $d$.
Let $X_{n,d}$ be the functor on affine $\FF_p$-schemes 
such that $X_{n,d}(\Spec R)$ is the set of invertible 
$W_n(R)$-matrices of rank $h$. Then
$X_{n,d}$ is an affine open subscheme of the
affine space of dimension $nh^2$ over $\FF_p$.
We define a morphism $\pi_{n,d}:X_{n,d}\to\Disp_{n,d}$
such that the truncated display $\pi_{n,d}(M)$ is given by 
the normal representation $(L,T,\Psi)$, where $L=W_{n}(R)^{h-d}$
and $T=W_{n}(R)^d$, and $M$ is the matrix representation
of $\Psi$. Let $G_{n,d}$ be the sheaf of groups 
such that $G_{n,d}(R)$ is the group of invertible matrices
$\left(\begin{smallmatrix}A&B\\C&D\end{smallmatrix}\right)$
with $A\in\Aut(L)$, $B\in\Hom(T,L)$, 
$C\in\Hom(L,I_{n+1,R}\otimes_{W_n(R)}T)$, and $T\in\Aut(T)$
for $L$ and $T$ as above. Then $G_{n,d}$ is an affine open
subscheme of the affine space of dimension $nh^2$ over
$\FF_p$. The morphism $\pi_{n,d}$ is a $G_{n,d}$-torsor.
Thus $\Disp_{n,d}$ and $\Disp_n$ are smooth algebraic 
stacks of dimension zero over $\FF_p$.

The truncation morphism $\tau_n$ is smooth and surjective
because it commutes
with the obvious projection $X_{n+1,d}\to X_{n,d}$,
which is smooth and surjective. 
The relative dimension of $\tau$ is the
difference of the dimensions of its source and target,
which are both zero.
\end{proof}

Let $\Disp\to\Spec\FF_p$ be the stack of displays
over $\FF_p$-algebras. 

\begin{Cor}
The fibered category $\Disp$ is an
affine algebraic stack over $\FF_p$, which has a
presentation $\pi:X\to\Disp$ such that $\pi$ and
the compositions $X\to\Disp\xrightarrow\tau\Disp_n$
for $n\ge 0$ are $\infty$-smooth coverings.
\end{Cor}

\begin{proof}
By Lemma \ref{Le-disp-trunc-disp}, $\Disp$ is the projective 
limit over $n$ of $\Disp_n$. Thus the corollary follows from
Proposition \ref{Pr-Disp-n-smooth} by the 
proof of Lemma \ref{Le-BT-pres}.
\end{proof}

\subsection{Nilpotent truncated displays}

Let $R$ be an $\FF_p$-algebra.
For each truncated display $\PPP=(P,Q,\iota,\varepsilon,F,F_1)$
of positive level $n$ over $R$ there is a unique homomorphism
$$
V^\sharp:P\to P^{(1)}=W_n(R)\otimes_{f,W_n(R)}P
$$
such that $V^\sharp(F_1(x))=1\otimes x$ for $x\in Q$.
If $F^\sharp:P^{(1)}\to P$ denotes the linearisation of $F$,
we have $F^\sharp V^\sharp=p$ and $V^\sharp F^\sharp=p$.
This is analogous to the case of displays; see
\cite[Lemma 10]{Zink-Disp}. 
The construction of $V^\sharp$ is compatible with truncation.
We call $\PPP$ nilpotent if for some $n$ the $n$-th 
iterate of $V^\sharp$
$$
P\to P^{(1)}\to\cdots\to P^{(n)}
$$
is zero. Since the ideal $I_{n,R}$ is nilpotent,
$\PPP$ is nilpotent if and only if the
truncation $\tau_1(\PPP)$ of level $1$ is nilpotent.
A display over $R$ is nilpotent if and only if all
its truncations are nilpotent.

\begin{Lemma}
\label{Le-Disp-o}
There are unique reduced closed substacks
$\Disp^o\subseteq\Disp$ and $\Disp^o_n\subseteq\Disp_n$
for $n\ge 1$ such that the geometric points of $\Disp^o$
and $\Disp^o_n$ are precisely the nilpotent (truncated)
displays. There is a Cartesian diagram
$$
\xymatrix@M+0.2em{
\Disp^o \ar[r] \ar[d] &
\Disp^o_{n+1} \ar[r] \ar[d] &
\Disp^o_{n} \ar[d] \\
\Disp \ar[r]^-\tau &
\Disp_{n+1} \ar[r]^-\tau &
\Disp_{n}.\!
}
$$
In particular, the closed immersion $\Disp^o\to\Disp$ is 
of finite presentation.
\end{Lemma}

\begin{proof}
Over a field, a truncated display of level $1$ and rank
$h$ is nilpotent if and only if the $h$-th iterate of
$V^\sharp$ vanishes. Thus for a display or truncated
display of positive level $\PPP$ over an $\FF_p$-algebra
$R$ the points of $\Spec R$ where $\PPP$ is nilpotent
form a closed subset. Since $\Disp$ and $\Disp_n$ have
reduced presentations and since the truncation morphisms
$\tau$ are reduced, the existence of the reduced closed
substacks $\Disp^o$ and $\Disp^o_n$
and the Cartesian diagram follow; cf.\ Lemma \ref{Le-BT-o}.
\end{proof}


\section{Smoothness of the truncated display functor}

\subsection{The truncated display functor}

We begin with the observation that the display
functors $\Phi_R$ induce truncated display functors
on each level. Recall that $(\pdivC_n/R)$ is
the category of truncated Barsotti-Tate 
groups of level $n$ over $R$, and $(\dispC_n/R)$ is
the category of truncated displays of level $n$ over $R$,
which is defined if $R$ is an $\FF_p$-algebra.

\begin{Prop}
\label{Pr-Phi-n-R}
For each $\FF_p$-algebra $R$ and each positive integer $n$
there is a unique functor 
$$
\Phi_{n,R}:(\pdivC_n/R)\to(\dispC_n/R)
$$ 
which is compatible with base change in $R$ and with the
truncation functors from $n+1$ to $n$ on both sides such 
that the system $(\Phi_{n,R})_{n\ge 1}$ induces $\Phi_R$ in the
projective limit. 
\end{Prop}

\begin{proof}
Let $(\pdivC_n/R)'$ be the category of all $G$ in $(\pdivC_n/R)$ 
which can be written as the kernel of an isogeny of $p$-divisible
groups $H_0\to H_1$ over $R$.
First we define a functor
$$
\Phi_{n,R}':(\pdivC_n/R)'\to(\predispC_n/R).
$$

For $G$ in $(\pdivC_n/R)'$ we choose an isogeny of $p$-divisible
group $H_0\to H_1$ with kernel $G$ and define
$$
\Phi_{n,R}'(G)=\Coker(\tau_n\Phi_R(H_0)\to\tau_n\Phi_R(H_1)),
$$
where $\Phi_R$ is given by Proposition \ref{Pr-disp-functor}, 
and where
$\tau_n$ is the truncation from displays to truncated 
displays of level $n$. If $g:G\to G'$ is a homomorphism
in $(\pdivC_n/R)'$ such that $G$ is the kernel of $H_0\to H_1$ 
and $G'$ is the kernel of $H_0'\to H_1'$,
we define $\Phi_{n,R}'(g)$ as follows.
Let $H_0''=H_0\times H_0'$, let $G\to H_0''$ be 
given by $(1,g)$, and let $H_1''=H_0''/G$. The projections
$H_0\leftarrow H_0''\to H_0'$ extend uniquely to homomorphisms 
of complexes $H_*\leftarrow H_*''\to H'_*$, where the first arrow 
is a quasi-isomorphism. This means that its cone is exact,
which is preserved by $\tau_n\circ\Phi_R$.
Thus the homomorphisms of complexes 
$$
\tau_n\Phi_R(H_*)\leftarrow\tau_n\Phi_R(H''_*)\to
\tau_n\Phi_R(H'_*)$$ 
induce a homomorphism of pre-displays
$\Phi_{n,R}'(g):\Phi_{n,R}'(G)\to\Phi_{n,R}'(G')$
on the cokernels.
It is easy to verify that $\Phi_{n,R}'$ is a well-defined 
functor, which is independent of the chosen isogenies;
see also \cite[8.5]{Lau-Frames} and \cite[4.1]{Lau-Relation}.

Since $\Phi_R$ and $\tau_n$ are compatible with base change 
in $R$, for a ring homomorphism $\alpha:R\to R'$ and for
$G$ in $(\pdivC_n/R)'$ we get a natural homomorphism of
pre-displays over $R'$
$$
u':W_n(R')\otimes_{W_n(R)}\Phi'_{n,R}(G)\to\Phi'_{n,R'}(G\otimes_RR').
$$
If $\alpha$ is ind-\'etale, Corollary 
\ref{Co-trunc-base-change-ind-et}
implies that $u'$ is an isomorphism.
Since $W_n$ preserves ind-\'etale coverings
of rings by \cite[Prop.~A.8]{Langer-Zink-dRW}; 
cf.\ Remark \ref{Re-ind-et},
ind-\'etale descent is effective for pre-displays
of level $n$. 

Assume that $G$ is the $p^n$-torsion of a $p$-divisible
group $H$ over $R$. Then we can use the isogeny
$p^n:H\to H$ in the construction of $\Phi'_{n,R}(G)$.
Since $p^n$ annihilates $W_n(R)$, it follows that 
$\Phi'_{n,R}(G)=\tau_n\Phi_{R}(H)$. In particular,
in this case the pre-display $\Phi'_{n,R}(G)$ is a
truncated display of level $n$. 

For each $G\in(\pdivC_n/R)$ 
there is a sequence of faithfully flat smooth ring
homomorphisms $R=R_0\to R_1\to R_2\cdots$ such that,
if we write $R'=\varinjlim R_i$, the group $G\otimes_RR'$ 
is the $p^n$-torsion of a $p$-divisible group over $R'$;
see Lemma \ref{Le-BT-pres}.
Since a surjective smooth morphism of schemes has a section
\'etale locally in the base, we can assume that $R\to R'$
is ind-\'etale. 
By Lemma \ref{Le-trunc-disp-local} it follows that the image 
of $\Phi_{n,R}'$ lies in $(\dispC_n/R)$, and by ind-\'etale
descent we get a unique extension of $\Phi'_{n,R}$
to a functor $\Phi_{n,R}$ 
as in the proposition which is compatible with ind-\'etale 
base change in $R$.

For an arbitrary homomorphism $\alpha:R\to R'$ of 
$\FF_p$-algebras and for $G$ in $(\pdivC_n/R)$, 
by ind-\'etale descent,
the above homomorphisms $u'$ induce a homomorphism
of pre-displays over $R'$
$$
u:W_n(R')\otimes_{W_n(R)}\Phi_{n,R}(G)\to\Phi_{n,R'}(G\otimes_RR').
$$
Since $\Phi_{n,R}(G)$ and $\Phi_{n,R'}(G\otimes R')$ are 
truncated displays, $u$ induces a base change homomorphism
of truncated displays over $R'$
$$
\tilde u:\alpha_*\Phi_{n,R}(G)\to\Phi_{n,R'}(G\otimes_RR').
$$
We claim that $\tilde u$ is an isomorphism.
If $G$ is the $p^n$-torsion of a $p$-divisible
group $H$ over $R$, this is true because then
$\Phi'_{n,R}(G)=\tau_n\Phi_R(H)$.
The general case follows by passing
to an ind-\'etale covering of $R$.
\end{proof}

\begin{Remark}
For the construction of $\Phi_{n,R}$ as a functor from
$(\pdivC_n/R)$ to $(\predisp_n/R)$ one can also use the
theorem of Raynaud \cite[3.1.1]{BBM} that a 
commutative finite flat group scheme can be embedded into
an Abelian variety 
locally in the base. However, an additional argument is
needed to ensure that the image of $\Phi_{n,R}$ consists
of truncated displays.
\end{Remark}

\begin{Remark}
\label{Re-Phi-n-perf}
If $k$ is a perfect field of characteristic $p$,
in view of Lemma \ref{Le-Dieud-disp-field},
the functor $\Phi_{n,k}:(\pdivC_n/k)\to(\dispC_n/k)$
is an equivalence of categories by classical
Dieudonn\'e theory.
\end{Remark}

\begin{Remark}
The definition of the dual display carries over to truncated
displays, and the functor $\Phi_{n,R}$ preserves duality
because this holds for $\Phi_R$; see Remark \ref{Re-Phi-dual}.
We leave out the details.
\end{Remark}

\subsection{Smoothness}

The functors $\Phi_{n,R}$ for variable $\FF_p$-algebras $R$ 
induce a morphism of algebraic stacks over $\FF_p$
$$
\phi_n:\overline{\BTst_n}\to\Disp_n.
$$
The source and target of $\phi_n$ are smooth over $\FF_p$
of pure dimension zero. For a perfect field $k$ of 
characteristic $p$ the functor
$$
\phi_n(k):\overline{\BTst_n}(k)\to\Disp_n(k)
$$ 
is an equivalence; see Remark \ref{Re-Phi-n-perf}.

\begin{Thm}
\label{Th-phi-n-smooth}
The morphism $\phi_n$ is smooth and surjective.
\end{Thm}

\begin{proof}
Let $\UUU\subset\overline{\BTst_n}$ be the open substack
where $\phi_n$ is smooth. We consider a geometric
point in $\overline{\BTst_n}(k)$ for an algebraically closed
field $k$, given by a truncated Barsotti-Tate group $G$
over $k$. The tangent space
$t_G(\overline{\BTst_n})$ is the set of isomorphism classes
of deformations of $G$ over $k[\varepsilon]$.
Let $\PPP=\phi(G)$. Since $\overline{\BTst_n}$
and $\Disp_n$ are smooth, $G$ lies in $\UUU(k)$ if and
only if $\phi_n$ induces a surjective map on tangent
spaces
$$
t_G(\phi_n):t_G(\overline{\BTst_n})\to t_{\PPP}(\Disp_n).
$$
There is a $p$-divisible
group $H$ over $k$ such that $G\cong H[p^n]$. Let 
$\tilde\PPP=\Phi_k(H)$ be the associated display,
thus $\PPP=\tau_n\tilde\PPP$.
We have a commutative square of tangent spaces
$$
\xymatrix@M+0.2em@C+1em{
t_H(\overline\BTst) \ar[r]^-{t_H(\phi)} \ar[d]_{t_H(\tau)} &
t_{\!\tilde\PPP}(\Disp) \ar[d]^{t_{\!\tilde\PPP}(\tau)} \\
t_G(\overline{\BTst_n}) \ar[r]^-{t_G(\phi_n)} &
t_\PPP(\Disp_n)
}
$$
where $\tau$ denotes the truncation morphisms and where
$\phi:\overline\BTst\to\Disp$ is induced by the functors
$\Phi_R$ for $\FF_p$-algebras $R$.
Here $t_{\!\tilde\PPP}(\tau)$ is surjective because the 
truncation morphisms $\Disp_{m+1}\to\Disp_m$
for $m\ge n$ are smooth.

If $G$ is infinitesimal or unipotent, $H$ is infinitesimal or
unipotent as well, and the map $t_H(\phi)$ is bijective by 
Corollary \ref{Co-Phi-form-etale}.
Thus $\UUU$ contains all infinitesimal and unipotent groups,
and $\UUU=\overline{\BTst_n}$ by Lemma \ref{Le-open-in-BT} below.
\end{proof}

\begin{Lemma}
\label{Le-open-in-BT}
Let $\UUU$ be an open substack of $\overline{\BTst_n}$ 
that contains all points which correspond to infinitesimal 
or unipotent groups. Then $\UUU=\overline{\BTst_n}$.
\end{Lemma}

\begin{proof}
For an algebraically closed field $k$ and
$G\in\overline{\BTst_n}(k)$ we have to show that $G$ lies in
$\UUU(k)$. We write $G=H[p^n]$ for a $p$-divisible
group $H$ over $k$. Let $K$ be an algebraic closure of $k((t))$
and let $R$ be the ring of integers of $K$. 
Let $\nu$ be the Newton polygon of $H$ and let $\beta$ be
the unique linear Newton polygon with $\beta\preceq\nu$.
By \cite[Thm.~3.2]{Oort-Newton+Formal} there is a $p$-divisible 
group $H''$ over $R$ with generic Newton polygon $\nu$ and special
Newton polygon $\beta$. Since $K$ is algebraically closed,
there is an isogeny $H''_K\to H\otimes_kK$. Let $C$ be its
kernel, let $C_R\subset H''$ be the schematic closure of $C$,
and let $H'=H''/C_R$. Then $H'_K\cong H\otimes_kK$, and 
the special fibre $H'_k$ is isoclinic.
We obtain a commutative diagram where $g$ is given by $G$,
and $g'$ is given by $H'[p^n]$. 
$$
\xymatrix{
\Spec K \ar[r] \ar[d] & \Spec k \ar[d]^g \\
\Spec R \ar[r]^{g'} & \overline{\BTst_n}.
}
$$
Here $g^{\prime -1}(\UUU)$ is an open subset of $\Spec R$, 
which contains the closed point since the special fibre
of $H'[p^n]$ is unipotent or infinitesimal. Thus
$g^{\prime -1}(\UUU)$ is all of $\Spec R$, which implies 
that $G$ lies in $\UUU(k)$.
\end{proof}

We consider the diagonal morphism
$$
\Delta:\overline{\BTst_n}\to
\overline{\BTst_n}\times_{\Disp_n}\overline{\BTst_n}
$$
and view it as a morphism over 
$\overline{\BTst_n}\times\overline{\BTst_n}$.
Let $X$ be a affine $\FF_p$-scheme. For
$g:X\to\overline{\BTst_n}\times\overline{\BTst_n}$,
corresponding to two truncated Barsotti-Tate
groups $G_1$ and $G_2$ over $X$, with associated truncated 
displays $\PPP_1$ and $\PPP_2$, the inverse image of 
$\Delta$ under $g$ is the morphism of affine $X$-schemes
$$
\uIsom(G_1,G_2)\to\uIsom(\PPP_1,\PPP_2).
$$
For $G\in\overline\BTst_n(X)$, 
with associated truncated display $\PPP$, let
$$
\uAut^o(G)=\Ker(\uAut G\to\uAut\PPP).
$$
This is an affine group scheme over $X$.
For varying $X$ and $G$
we obtain a relative affine group scheme
$\uAut^o(G^{\univ})$ over $\overline{\BTst_n}$.
Let
$$
\pi_1,\pi_2:\overline{\BTst_n}\times_{\Disp_n}\overline{\BTst_n}
\to\overline{\BTst_n}
$$
be the two projections. 

\begin{Thm}
\label{Th-phi-n-diag}
The representable affine morphism $\Delta$ is 
finite, flat, radicial, and surjective.
The group scheme $\uAut^o(G^{\univ})\to\overline{\BTst_n}$ 
is commutative and finite flat, and $\Delta$ is a torsor 
under $\pi_i^*\uAut^o(G^{\univ})$ for $i=1,2$.
\end{Thm}

\begin{proof}
We write $\XXX=\overline{\BTst_n}$ and $\YYY=\Disp_n$.
Let $\pi:X\to\XXX$ be a smooth presentation with affine $X$. 
We can assume that $X$ has pure dimension $m$, 
which implies that $\pi$ has pure dimension $m$.
By Theorem \ref{Th-phi-n-smooth}, the composition
$\psi=\phi_n\circ\pi:X\to\XXX\to\YYY$ is a smooth
presentation of pure dimension $m$ as well.
It follows that $X'=X\times_\XXX X$ and $Y'=X\times_\YYY X$
are smooth $\FF_p$-schemes of pure dimension $2m$.
The natural morphism $\phi':X'\to Y'$ can be identified
with the inverse image of $\Delta:\XXX\to\XXX\times_\YYY\XXX$ 
under the smooth presentation $X\times X\to\XXX\times\XXX$.

Since $\phi_n:\XXX\to\YYY$ is an equivalence on geometric points,
$\phi'$ is bijective on geometric points. Since $X'$ and $Y'$ are 
equidimensional, the irreducible components of $X'$ 
are in bijection to the irreducible components of $Y'$.
Thus $\phi'$ is flat; see \cite[Thm.~23.1]{Matsumura}.
Let $Z$ be the normalisation of $Y'$ in the purely inseparable 
extension of function fields defined by $X'\to Y'$. 
Then $Z\to Y'$ is bijective on geometric points, so
$X'\to Z$ is bijective on geometric points, and
$X'=Z$ by Zariski's main theorem. Thus $\phi'$ is 
finite, flat, radicial, and surjective, which implies that $\Delta$ 
is finite, flat, radicial, and surjective.

Recall that a morphism $T\to S$ with an action of
an $S$-group $A$ on $T$ is called a quasi-torsor if for
each $S'\to S$ the fibre $T(S')$ is either empty
or isomorphic to $A(S')$ as an $A(S')$-set.
Clearly $\Delta$ is a quasi-torsor under the obvious
right action of $\pi_1^*\uAut^o(G^{\univ})$
and under the obvious left action of 
$\pi_2^*\uAut^o(G^{\univ})$. 
Since $\Delta$ is finite, flat, and surjective, it
follows that the quasi-torsor $\Delta$ is a torsor,
and that $\pi_i^*\uAut^o(G^{\univ})$ is finite 
and flat over $\XXX\times_\YYY\XXX$. 
Since $\phi_n$ is smooth and surjective,
the same holds for the projections $\pi_i$, and it follows
that $\uAut^o(G^{\univ})$ is finite and flat over $\XXX$.

It remains to show that $\uAut^o(G^{\univ})$ is commutative.
It suffices to show this on a dense open substack of $\XXX$,
and thus is suffices to show that for 
$G=(\ZZ/p^n\ZZ)^r\times(\mu_{p^n})^s$ over a field $k$
the $k$-group scheme $\uAut^o(G)$ is commutative.
Now $\uHom(\mu_{p^n},\ZZ/p^n\ZZ)$ is zero, and
the group schemes $\uAut((\ZZ/p^n\ZZ)^r)$ and 
$\uAut((\mu_{p^n})^s)$ are \'etale. Since $\phi_n$
is an equivalence on geometric points, it follows that
$\uAut^o(G)$ is contained in the group scheme 
$\{\left(
\begin{smallmatrix}1&0\\a&1\end{smallmatrix}
\right)\mid a\in\mu_{p^n}^{rs}\}$,
which is commutative.
\end{proof}

\begin{Remark}
\label{Re-phi-n-diag}
For $G=(\ZZ/p^n\ZZ)^r\times(\mu_{p^n})^s$ as above,
$\uAut^o(G)$ is in fact equal to 
$\{\left(\begin{smallmatrix}1&0\\a&1\end{smallmatrix}
\right)\mid a\in\mu_{p^n}^{rs}\}$. 
Since ordinary groups are dense in $\overline{\BTst_n}$,
it follows that on the open and closed substack of 
$\overline{\BTst_n}$ where the universal group has dimension 
$s$ and codimension $r$, the degree of the finite flat
group scheme $\uAut^o(G^{\univ})$ is equal to $p^{rsn}$. 

To prove the first equality,
it suffices to show that the truncated displays 
$\PPP_1=\Phi_{n,\FF_p}(\ZZ/p^n\ZZ)$ and 
$\PPP_2=\Phi_{n,\FF_p}(\mu_{p^n})$ 
satisfy $\uHom(\PPP_1,\PPP_2)=0$.
For an $\FF_p$-algebra $R$, if $i:I_{n+1,R}\to W_n(R)$ 
denotes the natural homomorphism, we have
$$
\PPP_1=(W_n(R),W_n(R),\id,i,f,pf), 
$$
$$
\PPP_2=(W_n(R),I_{n+1,R},i,\id,f_1,f).
$$
Thus $\uHom(\PPP_1,\PPP_2)$ can be identified with the
set of $a\in I_{n+1,R}$ such that $f_1(a)=i(a)$, or
equivalently $a=v(i(a))$, which implies that $a=0$.
\end{Remark}

By passing to the limit,
Theorems \ref{Th-phi-n-smooth} and \ref{Th-phi-n-diag} 
give the following information on the morphism
$\phi:\overline\BTst\to\Disp$. 
For a $p$-divisible group $G$ over an $\FF_p$-algebra $R$
with associated display $\PPP$ let $\uAut^o(G)$ be the
kernel of $\uAut(G)\to\uAut(\PPP)$.
This is an affine group scheme over $R$, which is the
projective limit over $n$ of the finite flat group schemes
$\uAut^o(G[p^n])$; thus $\uAut^o(G)$ is commutative and flat.
Let $G^{\univ}$ be the universal $p$-divisible group and
let $\pi_1,\pi_2:\Disp\times_{\overline\BTst}\Disp\to\Disp$ be
the two projections.

\begin{Cor}
\label{Co-phi}
The morphism $\phi$ is faithfully flat, and its diagonal 
is a torsor under the flat affine group scheme 
$\pi_i^*\uAut^o(G^{\univ})$ for $i=1,2$. 
\qed
\end{Cor}

The limit $\uAut^o(G)=\varprojlim_n\uAut^o(G[p^n])$
can show quite different behaviour depending on $G$;
see Corollary \ref{Co-uAut-o} in the next section.


\section{Classification of formal $p$-divisible groups}

As an application of Theorems \ref{Th-phi-n-smooth} and
\ref{Th-phi-n-diag} together with 
Corollary \ref{Co-Phi-form-etale}
we will prove the following.

\begin{Thm}
\label{Th-formal-classif}
For each $p$-adic ring $R$, the functor $\Phi_R$ from
$p$-divisible group over $R$ to displays over $R$ induces
an equivalence
$$
\Phi_R^1:(\text{formal $p$-divisible groups}/R)\to
(\text{nilpotent displays}/R).
$$
\end{Thm}

It is known by \cite{Zink-Disp} and \cite{Lau-Disp} that
the functor $\BT_R$ defined in \cite{Zink-Disp} from displays
over $R$ to formal Lie groups over $R$ induces an equivalence 
between nilpotent displays over $R$ and formal $p$-divisible 
groups over $R$. The relation between $\Phi_R$ and $\BT_R$
is discussed in section \ref{Se-BT-functor}.

\begin{Lemma}
\label{Le-phi-o}
For $n\ge 1$ there is a Cartesian diagram of Artin algebraic stacks
$$
\xymatrix@M+0.2em{
\BTst_n^o \ar[r]^-{\phi_n^o} \ar[d] &
\Disp_n^o \ar[d] \\
\overline{\BTst_n} \ar[r]^-{\phi_n} &
\Disp_n
}
$$
where the vertical arrows are the immersions given
by Lemmas \ref{Le-BT-o} and \ref{Le-Disp-o}, and
where $\phi_n$ is given by the functors $\Phi_{n,R}$.
The projective limit over $n$ is a Cartesian diagram
of affine algebraic stacks
$$
\xymatrix@M+0.2em{
\BTst^o \ar[r]^-{\phi^o} \ar[d] &
\Disp^o \ar[d] \\
\overline{\BTst} \ar[r]^-{\phi} &
\Disp
}
$$
where $\phi$ is given by the functors $\Phi_R$.
\end{Lemma}

\begin{proof}
Since $\phi_n$ is smooth, the inverse image of
$\Disp_n^o$ under $\phi_n$ is a reduced closed
substack $\BTst_n'$ of $\overline{\BTst_n}$.
By classical Dieudonn\'e theory, the geometric
points of $\BTst_n'$ and of $\BTst_n^o$ coincide;
thus $\BTst^o_n=\BTst_n'$, and we get the first
Cartesian square. The projective limit
of $\BTst^o_n$ is $\BTst^o$ by Lemma \ref{Le-BT-o},
and the projective limit of $\Disp_n^o$ is $\Disp^o$
by Lemma \ref{Le-Disp-o}. Hence the second Cartesian 
square follows from the first one.
\end{proof}

The essential part of Theorem \ref{Th-formal-classif}
is the following result.

\begin{Thm}
\label{Th-phi-o}
The morphism $\phi^o:\BTst^o\to\Disp^o$ is an isomorphism.
\end{Thm}

\begin{proof}
We use the following notation.
$$
\XXX=\BTst^o,\qquad\XXX_n=\BTst_n^o, 
$$
$$
\YYY=\Disp^o,\qquad\YYY_n=\Disp_n^o.
$$
As in the proof of Lemma \ref{Le-BT-pres}
we choose smooth presentations $X_n\to\XXX_n$ with affine $X_n$
such that the truncation morphisms $\XXX_{n+1}\to\XXX_n$ 
lift to morphisms $X_{n+1}\to X_n$ where 
$X_{n+1}\to X_n\times_{\XXX_n}\XXX_{n+1}$ is smooth
and surjective.
Since $\XXX_n\to\YYY_n$ is smooth and surjective by 
Theorem \ref{Th-phi-n-smooth} and Lemma \ref{Le-phi-o}, 
the composition $X_n\to\XXX_n\to\YYY_n$ is a 
smooth presentation. 
Let $X=\varprojlim_n X_n$ and
$$
X_n'=X_n\times_{\XXX_n}X_n,\qquad X'=
X\times_\XXX X=\varprojlim_n X_n', 
$$
$$
Y_n'=X_n\times_{\YYY_n}X_n,\qquad\: 
Y'=X\times_\YYY X=\varprojlim_n Y'_n.
$$

Here $X\to\XXX$ is a faithfully flat presentation because 
for $Z\to\XXX$ with affine $Z$ we have 
$Z\times_\XXX X=\varprojlim_n (Z\times_{\XXX_n}X_n)$,
and a projective limit of faithfully flat affine $Z$-schemes is
a faithfully flat affine $Z$-scheme. Similarly, the compositoin
$X\to\XXX\to\YYY$ is a faithfully flat presentation.  
We have an infinite commutative diagram:
$$
\xymatrix@M+0.2em{
X'_1 \ar[d] &
X'_2 \ar[l] \ar[d] &
X'_3 \ar[l] \ar[d] &
\cdots \ar[l] \\
Y'_1 &
Y'_2 \ar[l] &
Y'_3 \ar[l] &
\cdots \ar[l]
}
$$
The theorem means that the limit $X'\to Y'$ is an isomorphism.

Let $G_n$ be the infinitesimal truncated Barsotti-Tate
group over $X_n$ which defines the presentation 
$X_n\to\XXX_n$ and let $\pi_n:Y_n'\to X_n$ be the
first projection. By Theorem \ref{Th-phi-n-diag},
$X'_n\to Y'_n$ is a torsor under the commutative 
infinitesimal finite flat
group scheme $A_n=\uAut^o(\pi_n^*G_n)$. 
The truncation induces a homomorphism of finite
flat group schemes over $Y'_{n+1}$
$$
\psi_n:A_{n+1}\to A_n\times_{Y'_n}Y'_{n+1}
$$
and a morphism 
$$
X'_{n+1}\to X'_n\times_{Y'_n}Y'_{n+1}
$$ 
which is equivariant with respect to $\psi_n$. 

\begin{Lemma}
\label{Le-m-n}
For each $m$ there is an $n\ge m$ such
that the transition homomorphism
$$
\psi_{m,n}:A_{n}\to A_m\times_{Y'_m}Y'_{n}
$$
is zero.
\end{Lemma}

If Lemma \ref{Le-m-n} is proved, it follows that there is a unique 
diagonal morphism which makes the following diagram commute.
$$
\xymatrix@M+0.2em{
X_m' \ar[d] &
X_n' \ar[l] \ar[d] \\
Y_m' &
Y_n' \ar[l] \ar[lu]  
}
$$
Thus $\varprojlim_n X'_n\to\varprojlim_n Y'_n$ is an
isomorphism, and Theorem \ref{Th-phi-o} follows. 

The essential case of Lemma \ref{Le-m-n} is the following.
Consider a geometric point $y:\Spec k\to Y'$ 
for an algebraically closed field $k$.
Let $A_{n,k}=A_n\times_{Y_n'}\Spec k$ and 
$Z_n=X_n'\times_{Y_n'}\Spec k$. 
Thus $Z_n$ is an $A_{n,k}$-torsor. We have homomorphisms of finite
$k$-group schemes $A_{n+1,k}\to A_{n,k}$ and equivariant morphisms
$Z_{n+1}\to Z_n$. Since $k$ is algebraically closed and
since $A_{n,k}$ is infinitesimal, $Z_n(k)$ has precisely one
element, and there are compatible isomorphisms $Z_n\cong A_{n,k}$.
For each $m$ the images of $A_{n,k}\to A_{m,k}$ for $n\ge m$ 
stabilise to a subgroup scheme $A_m'$ of $A_{m,k}$, and
$A'_{m+1}\to A'_m$ is an epimorphism.
Let $A'_n=\Spec B_n$ and $B=\varinjlim_n B_n$.

The geometric point $y:\Spec k\to Y'$ corresponds to two
formal $p$-divisible groups $G_1$ and $G_2$ over $k$ 
together with an isomorphism of the associated displays
$\alpha:\Phi_k(G_1)\cong\Phi_k(G_2)$. The $k$-scheme
$X'\times_{Y'}\Spec k$ can be identified with
$\varprojlim_n Z_n=\varprojlim_n A'_n=\Spec B$ and
classifies lifts of $\alpha$ to an isomorphism $G_1\cong G_2$. 
Thus Corollary \ref{Co-Phi-form-etale} implies that $B$ 
is a formally \'etale $k$-algebra. It follows that the 
cotangent complex $L_{B/k}$ has trivial homology in degrees 
$0$ and $1$; see Lemma \ref{Le-form-etale} below. 
The epimorphism $A_{n+1}'\to A_n'$ 
induces an injective homomorphism
$H_1(L_{B_n/k})\otimes_{B_n}B_{n+1}\to H_1(L_{B_{n+1}/k})$;
see \cite[Chap.~I, Prop.~3.3.4]{Messing-Crys}.
Since $L_{B/k}=\varinjlim_n L_{B_n/k}$ it follows that
$H_1(L_{B_n/k})$ is zero, which implies that the
finite infinitesimal group scheme $A_n'$ is zero
for each $n$.
Thus $X'\times_{Y'}\Spec k\cong\Spec k$.

Let us now prove Lemma \ref{Le-m-n}. First we note that 
$X_{n+1}'\to X_n'$ is smooth and surjective because 
this morphism can be factored as follows. 
$$
X_{n+1}\times_{\XXX_{n+1}}X_{n+1}\to
(X_n\times_{\XXX_n}\XXX_{n+1})
\times_{\XXX_{n+1}}
(X_n\times_{\XXX_n}\XXX_{n+1})
\to X_n\times_{\XXX_n}X_n.
$$
The first arrow is smooth and surjective by our assumptions 
on $X_n$, and the second arrow is smooth and surjective because 
this holds for $\XXX_{n+1}\to\XXX_n$ by Lemma \ref{Le-BT-o}.
Since $X'_n\to Y'_n$ is faithfully flat, it follows 
that $Y_{n+1}'\to Y'_n$ is faithfully flat as well.

Let $U_{m,n}\subset Y'_{n}$ be the open set of all points
$y$ such that the fibre $(\psi_{m,n})_y$ is non-zero. Since
$Y_n'$ is reduced, it suffices to show that for
each $m$ there is an $n$ such that $U_{m,n}$ is empty.
Assume that for some $m$ the set $U_{m,n}$ is non-empty
for all $n\ge m$. Let $V_{m,n}$ be the set of
generic points of $U_{m,n}$. Since $Y_{n+1}'\to Y_n'$
is flat, we have $V_{m,n+1}\to V_{m,n}$.
Since $V_{m,n}$ is finite and non-empty, the projective 
limit over $n$ of $V_{m,n}$ is non empty.
Hence there is a geometric point $y:\Spec k\to Y'$ such that 
$\Spec k\to Y'_n$ lies in $V_{m,n}$ for each $n\ge m$. 
If we perform the above construction for $y$, the group 
$A'_m$ is non-zero, which is impossible. 
Thus Lemma \ref{Le-m-n} and Theorem \ref{Th-phi-o} are proved.
\end{proof}

\begin{Lemma}
\label{Le-form-etale}
A ring homomorphism $\alpha:A\to B$ is formally \'etale if and only
if the cotangent complex $L_{B/A}$ has trivial homology in
degrees $0$ and $1$.
\end{Lemma}

\begin{proof}
Clearly $\alpha$ is formally unramified if and only if
$\Omega_{B/A}=H_0(L_{B/A})$ is zero. Let us assume that this 
holds. Since the obstructions to formal smoothness lie in
$\Ext^1_B(L_{B/A},M)$ for varying $B$-modules $M$, the implication
$\Leftarrow$ of the lemma is clear. 
So assume that $\alpha$ is formally smooth. We write $B=R/I$ 
for a polynomial ring $R$ over $A$; let $C=R/I^2$. 
Since $L_{R/A}$ is a free $R$-module in degree zero, 
the natural homomorphisms $H_1(L_{C/A})\to H_1(L_{C/R})$
and $H_1(L_{B/A})\to H_1(L_{B/R})$ are injective.
The homomorphism $H_1(L_{C/R})\to H_1(L_{B/R})$ can be
identified with $I^2/I^4\to I/I^2$, which is zero.
Thus $u:H_1(L_{C/A})\to H_1(L_{B/A})$ is zero as well.
Since $\alpha$ is formally smooth, $\id_B$ factors 
into $A$-algebra homomorphisms $B\to C\xrightarrow\pi B$, 
where $\pi$ is the projection.
Thus $u$ is surjective, and $H_1(L_{B/A})$ is zero.
\end{proof}

\begin{proof}[Proof of Theorem \ref{Th-formal-classif}]
We may assume that $p$ is nilpotent in $R$.
Since $\Phi_R$ is an additive functor, in order to
show that $\Phi_R^1$ is fully faithful it suffices 
to show that for two formal $p$-divisible groups 
$G_1$ and $G_2$ over $R$ with
associated nilpotent displays $\PPP_1$ and $\PPP_2$,
the map $\gamma:\Isom(G_1,G_2)\to\Isom(\PPP_1,\PPP_2)$
is bijective. We define an ideal $I\subset R$
by the Cartesian diagram
$$
\xymatrix@M+0.2em@C+1em{
\Spec R/I \ar[r] \ar[d] & \BTst^o\times\BTst^o \ar[d] \\
\Spec R \ar[r]^-{(G_1,G_2)} & \BTst\times\BTst
}
$$
If we write $G_1'=G_1\otimes_RR/I$ etc., 
$\Isom(G_1',G_2')\to\Isom(\PPP_1',\PPP_2')$ is
bijective by Theorem \ref{Th-phi-o}.
Since $\BTst^o\to\BTst$ is of finite presentation
by Lemma \ref{Le-BT-o} and
since $I$ is a nil-ideal, $I$ is nilpotent. By 
Corollary \ref{Co-Phi-form-etale} it follows that 
$\gamma$ is bijective.
We show that $\Phi_R^1$ is essentially surjective.
Since $\Disp^o\to\Disp$ is of finite presentation
by Lemma \ref{Le-Disp-o}, for a given nilpotent display $\PPP$ 
over $R$ we find a nilpotent ideal $I\subset R$ such
that the associated morphism $\Spec R/I\to\Spec R\to\Disp$ 
factors over $\Disp^o$. By Theorem \ref{Th-phi-o},
$\PPP_{R/I}$ lies in the image of $\Phi_{R/I}^1$.
Thus $\PPP$ lies in the image of $\Phi_R^1$ by
Corollary \ref{Co-Phi-form-etale}.
\end{proof}

\begin{Cor}
\label{Co-uAut-o}
Let $G$ be a $p$-divisible group over an $\FF_p$-algebra
$R$. The affine group scheme $\uAut^o(G)$ is
trivial if and only if for all $x\in\Spec R$ the fibre
$G_x$ is connected or unipotent.
\end{Cor}

\begin{proof}
Assume that $G$ is a $p$-divisible group over an algebraically
closed field $k$ such that $H=\QQ_p/\ZZ_p\oplus\mu_{p^\infty}$
is a direct summand of $G$. By Remark \ref{Re-phi-n-diag}, 
the group $\uAut^o(H[p^n])$ is isomorphic to $\mu_{p^n}$, 
with transition maps $\mu_{p^{n+1}}\to\mu_{p^n}$ given by
$\zeta\mapsto\zeta^p$. Thus $\uAut^o(H)$ is non-trivial,
which implies that $\uAut^o(G)$ is non-trivial as well.
This proves the implication $\Rightarrow$.

Assume now that all fibres of $G$ over $R$ are connected
or unipotent. Let $\Spec R/I$ and $\Spec R/J$ be the
inverse images of $\BTst^o$ under the morphisms 
$\Spec R\to\BTst$ defined by $G$ and by $G^\vee$.
Then $I$ and $J$ are finitely generated by Lemma \ref{Le-BT-o}, 
and $IJ$ is a nilideal, thus $IJ$ is nilpotent. 
In order to show that $\uAut^o(G)$ is trivial we may replace
$R$ by $R/IJ$. Then $R\to R/I\times R/J$ is injective.
Since $\uAut^o(G)$ is flat over $R$ by Corollary \ref{Co-phi},
we may further replace $R$ by $R/I$ and by $R/J$.
Since $\uAut^o(G)\cong\uAut^o(G^\vee)$ by Remark \ref{Re-Phi-dual},
in both cases the assertion follows from Theorem \ref{Th-phi-o}.
\end{proof}


\section{Dieudonn\'e theory over perfect rings}

The results of this section were obtained earlier
by Gabber by a different method.
For a perfect ring $R$ of characteristic $p$
we consider the Dieudonn\'e ring 
$$
D(R)=W(R)\{F,V\}/J
$$ 
where $W(R)\{F,V\}$ is the non-commutative polynomial ring 
in two variables over $W(R)$ and 
where $J$ is the ideal generated by the relations $Fa=f(a)F$ 
and $aV=Vf(a)$ for $a\in W(R)$, and $FV=VF=p$.

\begin{Defn} 
A \emph{projective Dieudonn\'e module over $R$} is a 
$D(R)$-module which is a projective $W(R)$-module of finite type. 
A \emph{truncated Dieudonn\'e module of level $n$ over $R$} is a 
$D(R)$-module $M$ which is a projective $W_n(R)$-module
of finite type; if $n=1$ we require also that 
$\Ker(F)=\Image(V)$ and $\Ker(V)=\Image(F)$ and that
$M/VM$ is a projective $R$-module.
An \emph{admissible torsion $W(R)$-module} is a finitely
presented $W(R)$-module
of projective dimension $\le 1$ which is annihilated 
by a power of $p$.
A \emph{finite Dieudonn\'e module over $R$} is a $D(R)$-module
which is an admissible torsion $W(R)$-module.
\end{Defn}

We denote by $(\DieudC/R)$, by $(\DieudC_n/R)$, and by
$(\DieudC^f/R)$ the categories of projective, truncated
level $n$, and finite Dieudonn\'e modules over $R$,
respectively.
For a homomorphism of perfect rings $R\to R'$ the scalar
extension by $W(R)\to W(R')$ induces functors from
projective, truncated, or finite Dieudonn\'e modules over
$R$ to such modules over $R'$. 

\begin{Lemma}
\label{Le-Dieud-trunc-1}
If $M$ is a projective or truncated Dieudonn\'e module of 
level $\ge 2$ then $M/pM$ is a truncated Dieudonn\'e module 
of level\/ $1$.
\end{Lemma}

\begin{proof}
We can replace $M$ by $M/p^2M$. The operators $F$ and $V$
of $M$ induce operators $\bar F$ and $\bar V$ of $\bar M=M/pM$.
Assume that $\bar F(\bar x)=0$ for an element $x\in M$. Then 
$F(x)\in pM$ and thus $F(x)=F(V(y))$ for some
$y\in M$. Hence $px=pV(y)$ and thus $\bar x=\bar V(\bar y)$.
This shows that $\Ker(\bar F)=\Image(\bar F)$, and similarly
$\Ker(\bar V)=\Image(\bar V)$. Since these relations remain 
true after base change, for each point of $\Spec R$
the dimensions of the fibres of $M/VM$ and 
of $M/FM$ add up to the dimension of the fibre of $\bar M$.
Thus the dimension of the fibre of $M/VM$ is an upper and
lower semicontinuous function on $\Spec R$, and
$M/VM$ is a projective $R$-module since $R$ is reduced.
\end{proof}

\begin{Lemma}
\label{Le-Dieud-disp}
Truncated Dieudonn\'e modules of level $n\ge 1$ over $R$
are equivalent to truncated displays of level $n$ over $R$,
and projective Dieudonn\'e modules over $R$ are equivalent
to displays over $R$.
\end{Lemma}

\begin{proof}
This extends Lemma \ref{Le-Dieud-disp-field}.
Since multiplication by $p$ is a isomorphism
$W_n(R)\cong I_{n+1,R}$ and since truncated displays
have normal decompositions by Lemma \ref{Le-pair-norm-dec},
truncated displays of level $n$ over $R$ are equivalent 
to quintuples $\PPP=(P,Q,\iota,\epsilon,F_1)$ where
$P$ and $Q$ are projective $W_n(R)$-modules of finite type
with homomorphisms $P\xrightarrow\epsilon Q\xrightarrow\iota P$
such that $\epsilon\iota=p$ and $\iota\epsilon=p$,
and where $F_1:Q\to P$ is a bijective $f$-linear map,
such that 

($\star$) {} $\Coker(\iota)$ is projective over $R$, and
$Q\xrightarrow\iota P\xrightarrow
{p^{n-1}\epsilon}Q\xrightarrow\iota P$ is exact.

\medskip
\noindent
By the proof of Lemma \ref{Le-Dieud-trunc-1}, condition ($\star$) 
is automatic if $n\ge 2$. It follows that truncated displays
of level $n\ge 1$ are equivalent to truncated Dieudonn\'e
modules of level $n$ by $\PPP\mapsto(P,F,V)$ 
with $F=F_1\epsilon$ and $V=\iota F_1^{-1}$.
The equivalence between displays and projective Dieudonn\'e 
modules follows easily; see also \cite[Lemma 2.4]{Lau-Relation}.
\end{proof}

\begin{Thm}
\label{Th-Gabber-trunc}
For a perfect ring $R$ of characteristic $p$, the functors 
$$
\Phi_{n,R}:(\pdivC_n/R)\to(\DieudC_n/R), 
$$
$$
\Phi_{R}:(\pdivC/R)\to(\DieudC/R)
$$
are equivalences; these functors are well-defined by
Lemma \ref{Le-Dieud-disp}.
\end{Thm}

Here $\Phi_R$ and $\Phi_{n,R}$ are defined
by $G\mapsto(\DD(G)_{W(R)/R},F,V)$, where $\DD(G)$
is the covariant Dieudonn\'e crystal of $G$,
and where $F$ and $V$ are induced by the Verschiebung
$G^{(1)}\to G$ and Frobenius $G\to G^{(1)}$.

\begin{proof}
For an $\FF_p$-algebra $A$, the perfection $A^{\per}$
is the direct limit of Frobenius $A\to A\to\cdots\,$.
For an $\FF_p$-scheme $X$, the perfection $X^{\per}$ is the
projective limit of Frobenius $X\leftarrow X\leftarrow\cdots\,$.
This is a local construction, which coincides with the
perfection of rings in the case of affine schemes.

Since the functor $\Phi_{n,R}$ is additive, 
it is fully faithful if 
for two groups $G_1$ and $G_2$ in $(\pdivC_n/R)$ with 
associated truncated displays $\PPP_1$ and $\PPP_2$, the 
map $\gamma:\Isom(G_1,G_2)\to\Isom(\PPP_1,\PPP_2)$ induced
by $\Phi_{n,R}$ is bijective. The morphism of 
affine $R$-schemes $\uIsom(G_1,G_2)\to\uIsom(\PPP_1,\PPP_2)$
is a torsor under an infinitesimal finite flat group 
scheme by Theorem \ref{Th-phi-n-diag}.
Thus $\uIsom(G_1,G_2)^{\per}\to\uIsom(\PPP_1,\PPP_2)^{\per}$ 
is an isomorphism, which implies that $\gamma$ is bijective.

Let us show that $\Phi_{n,R}$ is essentially surjective.
If $X\to\BTst_n$ is a smooth presentation with affine $X$, 
the composition $X\to\BTst_n\to\Disp_n$ is smooth and surjective 
by Theorem \ref{Th-phi-n-smooth}. For a given truncated
display $\PPP$ of level $n$ over $R$ let
$\Spec S=X\times_{\Disp_n}\Spec R$ and let $R'=S^{\per}$.
Then $R'$ is a perfect faithfully flat $R$-algebra such that 
$\PPP_{R'}$ lies in the image of $\Phi_{n,R'}$.
Since $R''=R'\otimes_RR'$ is perfect, the functors $\Phi_{n,R'}$
and $\Phi_{n,R''}$ are fully faithful. 
Thus $\PPP$ lies in the image of $\Phi_{n,R}$
by faithfully flat descent.

By passing to the projective limit it follows that
$\Phi_R$ is an equivalence.
\end{proof}

We denote by $(\pgrpC/R)$ the category of commutative
finite flat group schemes of $p$-power order over $R$.

\begin{Cor}
\label{Co-Gabber}
The covariant Dieudonn\'e crystal defines
a functor 
$$
\Phi_R^f:
(\pgrpC/R)\to(\DieudC^f/R)
$$
which is an equivalence of categories.
\end{Cor}

Over perfect field this is classical, over perfect valuation 
rings the result is proved in \cite{Berthelot-Parfait},
and the general case was first proved by Gabber 
by a reduction to the case of valuation rings.
Theorem \ref{Th-Gabber-trunc} is an immediate consequence
of Corollary \ref{Co-Gabber}.

The functor $\Phi_R^f$ can be defined by
$G\mapsto(\DD(G)_{W(R)},F,V)$ as above.

\begin{proof}[Proof of Corollary \ref{Co-Gabber}]
By a standard construction, 
the functor $\Phi_R$ and its inverse $\Phi_R^{-1}$ induce
formally a functor $\Phi_R^f$ as in Corollary \ref{Co-Gabber}
and a functor $\Psi_R^f$ in the opposite direction, which
are mutually inverse; this new definition of $\Phi_R^f$
coincides with the previous one by the construction of $\Phi_R$.

Let us explain this in more detail.
Since Zariski descent is effective for finite flat group schemes
and for finite Dieudonn\'e modules, 
in order to define $\Phi_R^f$ and $\Psi_R^f$
we may always pass to an open cover of $\Spec R$.
For each group $G$ in $(\pgrpC/R)$,
by a theorem of Raynaud \cite[Thm.~3.1.1]{BBM} 
there is an open cover of $\Spec R$ where $G$ can be written 
as the kernel of an isogeny of $p$-divisible groups $H_0\to H_1$. 
We define $\Phi_R^f(G)=\Coker[\Phi_R(H_0)\to\Phi_R(H_1)]$. 
This is independent of the chosen isogeny, functorial in $G$,
and compatible with localisations in $R$; see the proof of
Proposition \ref{Pr-Phi-n-R}. 

A homomorphism of projective Dieudonn\'e modules $u:N_0\to N_1$ 
over $R$ is called an isogeny if $u$ becomes bijective when $p$ 
is inverted. Then $u$ is injective, and its cokernel is
a finite Dieudonn\'e module $M$. In this case we define 
$\Psi_R^f(M)=\Ker[\Phi_R^{-1}(N_0)\to\Phi_R^{-1}(N_1)]$.
This depends only on $M$, the construction is functorial 
in $M$, compatible with localisations in $R$, and inverse to 
$\Phi_R^f$ when it is defined.

It remains to show that each finite Dieudonn\'e module 
$M$ over $R$ can be written as the cokernel of an isogeny of 
projective Dieudonn\'e modules locally in $\Spec R$. 
It is easy to find a commutative diagram 
$$
\xymatrix@M+0.2em{
N_1 \ar[r]^\epsilon \ar[d]^\pi &
Q \ar[r]^\iota \ar[d]^\psi &
N_1 \ar[d]^\pi \\
M \ar[r]^F & M \ar[r]^V & M
}
$$
where $Q$ and $N$ are free $W(R)$-modules of the same
finite rank, where $\iota,\epsilon,\pi$ are $W(R)$-linear
maps, and where $\psi$ is an $f$-linear map,
such that $\iota\epsilon=p$ and $\epsilon\iota=p$
and $\pi,\psi$ are surjective. The kernel of $\pi$
is a projective $W(R)$-module $N_0$ of finite type.
If we find a bijective $f$-linear map $F_1:Q\to N_1$ 
with $\pi F_1=\psi$, we can define $F=F_1\epsilon$ 
and $V=\iota F_1^{-1}$, and $M$ is the cokernel of the
isogeny if Dieudonn\'e modules $N_0\to N_1$.
Thus it suffices to show that $F_1$ exists locally in
$\Spec R$, which is a easy application of Nakayama's lemma.
\end{proof}


\section{Small presentations}
\label{Se-small-present}

In addition to the infinite dimensional presentation 
of $\BTst$ constructed in Lemma \ref{Le-BT-pres} one
can also find presentations where the covering space
is noetherian, or even of finite type. This will be
used in section \ref{Se-BT-functor}.
Assume that $G$ is a $p$-divisible group over
a $\ZZ_p$-algebra $A$. It defines a morphism
$$
\pi:\Spec A\to\BTst\times\Spec\ZZ_p.
$$
For each positive integer $m$ we also consider the restriction
$$
\pi^{(m)}:\Spec A/p^mA\to\BTst\times\Spec\ZZ/p^m\ZZ.
$$
We recall that the points of $\overline\BTst$ are pairs $(k,H)$
where $H$ is a $p$-divisible group over a
field $k$ of characteristic $p$, 
modulo the minimal equivalence relation such that
$(k,H)\sim(k',H')$ if there is 
$k\to k'$ with $H_{k'}\cong H'$.

\begin{Prop}
\label{Pr-AG-exist}
There is a pair $(A,G)$ with the following properties:
\begin{enumerate}
\renewcommand{\theenumi}{\roman{enumi}}
\item
The ring $A$ is $p$-adic and excellent.
\item
For each maximal ideal $\Fm$ of $A$, the residue field
$A/\Fm$ is perfect, and the group
$G\otimes\hat A_\Fm$ is a universal deformation of its
special fibre.
\item
All points of\/ $\overline\BTst$ 
which correspond to isoclinic $p$-divisible
groups lie in the image of $\pi$.
\end{enumerate}
\end{Prop}

\begin{Prop}
\label{Pr-AG-present}
Assume that $(A,G)$ satisfies (i) and (ii). Then $\pi$ is 
ind-smooth, and $\pi^{(m)}$ is ind-smooth and quasi-\'etale.
If (iii) holds as well, the morphisms
$\pi$ and $\pi^{(m)}$ are also faithfully flat.
\end{Prop}

We note the following consequence of
Propositions \ref{Pr-AG-exist} and \ref{Pr-AG-present}.

\begin{Cor}
\label{Co-small-present}
There is a presentation $\pi:\Spec A\to\BTst\times\Spec\ZZ_p$ 
such that $A$ is an excellent $p$-adic ring, $\pi$ ind-smooth,
and the residue fields of the maximal ideals of $A$ are perfect.
\qed
\end{Cor}

Let us prove Proposition \ref{Pr-AG-exist}.
As explained in \cite[Sec.~2]{NVW}, pairs $(A,G)$ that satisfy
(i)--(iii) can be constructed using integral models of 
suitable PEL-Shimura varieties. In that case, 
$A/p^nA$ is smooth over $\ZZ/p^n\ZZ$.
For completeness we give another construction using 
the Rapoport-Zink isogeny spaces of $p$-divisible groups.

\begin{proof}[Proof of Proposition \ref{Pr-AG-exist}]
Let $\GG$ be a $p$-divisible group over $\FF_p$ which is 
descent in the sense of \cite[Def.~2.13]{Rapoport-Zink}.
By \cite[Thm.~2.16]{Rapoport-Zink} there is a formal
scheme $M$ over $\Spf\ZZ_p$ which is locally
formally of finite type and which represents
the following functor on the category of rings
$R$ in which $p$ is nilpotent: $M(R)$ is the set
of isomorphism classes of pairs $(G,\rho)$ where
$G$ is a $p$-divisible group over $R$ and where
$\rho:\GG\otimes R/pR\to G\otimes_RR/pR$ is a quasi-isogeny. 
Let $G^{\univ}$ be the universal group over $M$.

If $U=\Spf A$ is an affine open subscheme of $M$,
then $A$ is an $I$-adic noetherian $\ZZ_p$-algebra
such that $A/I$ is of finite type over $\FF_p$.
Thus $A$ is $p$-adic and excellent; see \cite[Thm.~9]{Valabrega}.
The restriction of $G^{\univ}$ to $U$ defines
a $p$-divisible group $G$ over $\Spec A$ because
$p$-divisible groups over $\Spf A$ and over $\Spec A$ 
are equivalent; see \cite[Lemma 4.16]{Messing-Crys}.
For each maximal ideal $\Fm$ of $A$ the residue field
$A/\Fm$ is finite, and the definition of $M$ implies that
$G\otimes\hat A_\Fm$ is a universal deformation.
Thus (i) and (ii) hold.

For $0\le d\le h$ let $\GG_d$
be a decent isoclinic $p$-divisible group 
of height $h$ and dimension $d$ over $\FF_p$, 
and let $M_d$ be the associated isogeny space
over $\Spf\ZZ_p$.
It is well-known that there is a positive integer $\delta$
depending on $h$ such that if two $p$-divisible groups
of height $h$ over an algebraically closed field $k$ are
isogeneous, then there is an isogeny between them of
degree at most $\delta$. Thus there is a 
finite set of affine open subschemes $\Spf A_{d,i}$
of $M_d$ such that each isoclinic $p$-divisible group
of height $h$ over $k$ appears in the universal group 
over $\Spec A_{d,i,\red}$ for some $(d,i)$.
Let $A$ be the product of all $A_{d,i}$ and let
$G$ be the $p$-divisible group over $A$ defined
by the universal groups over $A_{d,i}$. Then
$(A,G)$ satisfies (i)--(iii).
\end{proof}

Let us prove the first part of Proposition \ref{Pr-AG-present}.

\begin{Lemma}
\label{Le-pi-ind-smooth}
If $(A,G)$ satisfies (i) and (ii), then $\pi$ is ind-smooth.
\end{Lemma}

\begin{proof}
Let $X=\Spec A$ and $\XXX=\BTst\times\Spec\ZZ_p$ and
$\XXX_n=\BTst_n\times\Spec\ZZ_p$.
For $T\to\XXX$ with affine $T$, the affine scheme 
$X\times_\XXX T$ is the projective limit of the 
affine schemes $X\times_{\XXX_n}T$. Since ind-smooth 
homomorphisms are stable under direct limits of rings,
to prove the lemma it suffices to show that the composition
$$
\pi_n:X\xrightarrow\pi\XXX\to\XXX_n
$$ 
is ind-smooth for each $n\ge 1$. By Popescu's theorem
this means that $\pi_n$ is regular. By EGA IV, 6.8.3 
this holds if and only if for each closed point $x\in X$
the composition
$$
\hat X=\Spec\hat\OOO_{X,x}\to X\to\XXX_n
$$
is regular. Let $k$ be the finite residue field of $x$.
Let $Y_0\to\XXX_n$ be a smooth presentation such that
$\pi_n(x)$ lifts to a $k$-valued point $y_0$ of $Y_0$.
Let $Y=Y_0\otimes_{\ZZ_p}W(k)$ and choose a closed point 
$y\in Y$ over $y_0$ with residue field $k$. 
Let $\hat Y=\Spec\hat\OOO_{Y,y}$.
Since $Y\to\XXX_n$ is smooth, $\hat Y\to\XXX_n$ is regular.
There is a $p$-divisible
group $H$ over $\hat Y$ such that the special fibre $H_y$
is equal to $G_x$ and such that the truncation $H[p^n]$ is the 
inverse image of the universal group over $\XXX_n$;
see \cite[Thm.~4.4]{Illusie-BT}. 
There is a unique morphism $\psi:\hat Y\to\hat X$ such that
$\psi^*G_{\hat X}$ is equal to $H$ as deformations of $G_x$
because $G_{\hat X}$ is assumed to be universal.
$$
\xymatrix@M+0.2em{
\hat X \ar[r] & X \ar[r]^G & \XXX \ar[r] & \XXX_n \\
& \hat Y \ar[ul]^\psi \ar[rr] \ar[ur]_H && Y \ar[u]
}
$$
Since $\XXX_n$ is smooth over $\ZZ_p$ by \cite[Thm.~4.4]{Illusie-BT}, 
$X$ and $Y_0$ are smooth over $\ZZ_p$. 
Thus $\hat\OOO_{X,x}$ and $\hat\OOO_{Y,y}$ are power series
rings over $W(k)$. The diagram induces a diagram of the tangent
spaces at the images of the closed point of $\hat Y$.
Here $\hat X\to\XXX$ is bijective on the tangent spaces 
since $G_{\hat X}$ is a universal deformation, $\XXX\to\XXX_n$
is bijective on the tangent spaces by \cite[Thm.~4.4]{Illusie-BT},
and $\hat Y\to\XXX_n$ is surjective on the tangent spaces since
$Y\to\XXX_n$ is smooth. Thus
$\psi$ is surjective on the tangent spaces, which implies that 
$\hat\OOO_{Y,y}$ is a power series
ring over $\hat\OOO_{X,x}$; in particular $\psi$ is
faithfully flat. Since $\hat Y\to\XXX_n$ is regular
it follows that $\hat X\to\XXX_n$ is regular.
\end{proof}

Now we prove the second part of Proposition \ref{Pr-AG-present}.
This is not used later.

\begin{Lemma}
\label{Le-pi-quasi-etale}
If $(A,G)$ satisfies (i) and (ii), then $\pi^{(m)}$ is quasi-\'etale.
\end{Lemma}

\begin{proof}
Let $X^{(m)}=\Spec A/p^m A$ and 
$\XXX^{(m)}=\BTst\times\Spec\ZZ/p^m\ZZ$. For a morphism
$Y\to\XXX^{(m)}$ with affine $Y$ let $Z=X\times_\XXX Y$
and let $\pi':Z\to Y$ be the second projection.
We have to show that $L_{Z/Y}$ is acyclic.
Here $\pi'$ is ind-smooth by Lemma \ref{Le-pi-ind-smooth}.
Thus $L_{Z/Y}$ is isomorphic to $\Omega_{Z/Y}$, and it suffices 
to show that $\pi^{(m)}$ is formally unramified. 

For an arbitrary
morphism $g:\Spec B\to\XXX^{(m)}$, given by a
$p$-divisible group $H$ over $B$, we write
$\Lambda_H=\Lie(H)\otimes_B\Lie(H^\vee)$.
There is a Kodaira-Spencer homomorphism
$\kappa_H:\Lambda_H\to\Omega_B$, which is surjective
if and only if $g$ is formally unramified. 
For a closed point $x\in X^{(m)}$ let $A_x$ be the
complete local ring at $x$ and let $\hat X^{(m)}=\Spec A_x$.
We have $i:\hat X^{(m)}\to X^{(m)}$.
There is a commutative diagram
$$
\xymatrix@M+0.2em{
0 \ar[r] &
i^*\Omega_{X^{(m)}} \ar[r] \ar[d]^{i^*\kappa_G} &
\Omega_{\hat X^{(m)}} \ar[r] \ar[d]^{\kappa_{i^*G}} &
\Omega_{\hat X^{(m)}/X^{(m)}} \ar[r] & 0 \\
& i^*\Lambda_G \ar@{=}[r] & \Lambda_{i^*G}.
}
$$

The upper line is exact because $\hat X^{(m)}\to X^{(m)}$ 
is regular, thus ind-smooth, which implies that 
$L_{\hat X^{(m)}/X^{(m)}}$ is concentrated in degree zero.
Since $A_x$ is isomorphic to a power series ring
$W_m(k)[[t_1,\ldots,t_r]]$ for a perfect field $k$, the $A_x$-module 
$\Omega_{\hat X^{(m)}}$ is free with basis $dt_1,\ldots,dt_r$.
These elements appear in $i^*\Omega_{X^{(m)}}$,
and thus $\Omega_{\hat X^{(m)}/X^{(m)}}$ is zero.
The homomorphism $\kappa_{i^*G}$ is an isomorphism
because $i^*G$ is assumed to be universal. 
Thus $i^*\kappa_G$ is an isomorphism for all $x$,
which implies that $\kappa_G$ is an isomorphism,
and $\pi^{(m)}$ is formally unramified as desired.
\end{proof}

Finally we prove the last part of Proposition \ref{Pr-AG-present}.

\begin{Lemma}
\label{Le-pi-surjective}
Assume that $(A,G)$ satisfies (i) and (ii).
Then $\pi$ is surjective if and only if (iii) holds.
\end{Lemma}

\begin{proof}
It suffices to prove the implication $\Leftarrow$.
We write $X=\Spec A$ and $\XXX=\BTst\times\Spec\ZZ_p$.
Let $k$ be an algebraically closed field of characteristic
$p$ and let $H:\Spec k\to\XXX$ be a geometric point, 
i.e.\ a $p$-divisible group $H$ over $k$.
It suffices to show that $X_H=X\times_\XXX\Spec k$ is non-empty.
Let $K$ be an algebraic closure of $k((t))$ and let
$R$ be the ring of integers in $K$. 
As in the proof of Lemma \ref{Le-open-in-BT} we find
a $p$-divisible group $H'$ over $R$ with generic fibre
$H\otimes_kK$ and with isoclinic special fibre.
We consider the morphism $\Spec R\to\XXX$ defined by $H'$.
The projection $X\times_\XXX\Spec R\to\Spec R$
is flat because it is ind-smooth by Lemma 
\ref{Le-pi-ind-smooth}, and its image contains the 
closed point of $\Spec R$ by (iii).
Thus the projection is surjective,
which implies that $X\times_\XXX\Spec K=X_H\otimes_kK$ 
is non-empty; thus $X_H$ is non-empty.
\end{proof}

\begin{proof}[Proof of Proposition \ref{Pr-AG-present}]
Use Lemmas \ref{Le-pi-ind-smooth},
\ref{Le-pi-quasi-etale}, and \ref{Le-pi-surjective}.
\end{proof}


\section{Relation with the functor BT}
\label{Se-BT-functor}

For a $p$-adic ring $R$
we consider the following commutative diagram of categories, 
where f.=formal, g.=groups, n.=nilpotent.
The vertical arrows are the inclusions.
The functor $\BT_R$ is defined in \cite[Thm.~81]{Zink-Disp}.
Its restriction to nilpotent displays gives
formal $p$-divisible groups by \cite[Cor.~89]{Zink-Disp}.
$$
\xymatrix@M+0.2em{
(\pdivC/R) \ar[r]^-{\Phi_R} &
(\dispC/R) \ar[r]^-{\BT_R} &
(\fgC/R) \\
(\fpdivC/R) \ar[r]^-{\Phi^1_R} \ar[u] &
(\ndispC/R) \ar[r]^-{\BT^1_R} \ar[u] &
(\fpdivC/R) \ar[r]^-{\Phi^1_R} \ar[u] &
(\ndispC/R)
}
$$
Here $\BT_R^1$ is an equivalence by \cite{Zink-Disp} 
if $R$ is excellent and by \cite{Lau-Disp} in general,
and $\Phi_R^1$ is an equivalence by 
Theorem \ref{Th-formal-classif}.

\begin{Lemma}
\label{Le-Phi-BT}
There is a natural isomorphism of functors
$\Phi_R^1\circ\BT_R^1\cong\id$.
\end{Lemma}

We have a functor
$\Upsilon_{\!R}:(\dispC/R)\to
(\text{filtered $F$-$V$-modules over }\WWW_R)$.

\smallskip

Let $\Upsilon_{\!R}^1:(\ndispC/R)\to
(\text{filtered $F$-$V$-modules over }\WWW_R)$
be its restriction.

\begin{proof}[Proof of Lemma \ref{Le-Phi-BT}]
By \cite[Thm.~94 and Cor.~97]{Zink-Disp}, for each nilpotent 
display $\PPP$ over $R$ there is an isomorphism,
functorial in $\PPP$ and in $R$,
$$
u_R(\PPP):\Upsilon_{\!R}^1(\Phi_R^1(\BT_R^1(\PPP)))
\cong\Upsilon_{\!R}^1(\PPP).
$$
We have to show that $u_R(\PPP)$ commutes with $F_1$.
This is automatic if $R$ has no $p$-torsion because
then $W(R)$ has no $p$-torsion, and $pF_1=F$.
Since $\Phi_R^1$ is an equivalence, we may assume
that $\PPP=\Phi_R^1(G)$ for a formal $p$-divisible
group $G$ over $R$. We may also assume that $p$ is
nilpotent in $R$.

Let $\Spec A\to\BTst\times\Spec\ZZ_p$ be a presentation 
given by a $p$-divisible group $H$ over $A$
such that $A$ is noetherian;
see Corollary \ref{Co-small-present}.
Let $J\subset A$ be the ideal such that
$X\times_{\BTst}\BTst^o=\Spec A/J$ and
let $\hat A$ be the $J$-adic completion of $A$.
If $A$ is constructed using isogeny spaces as in the
proof of Proposition \ref{Pr-AG-exist}, 
then $A$ is $I$-adic for an ideal 
$I$ which contains $J$, and thus $A$ is already $J$-adic.
In any case, since $\ZZ_p\to A\to\hat A$ is flat,
$\hat A$ has no $p$-torsion. Thus
the compatible system $[u_{A/J^n}(\Phi_{A/J^n}^1(H))]_{n\ge 1}$ 
necessarily preserves $F_1$.

For $G$ over $R$ as above we consider 
$\Spec R\times_{\BTst}\Spec A=\Spec S$.
Since $\Spec R_{\red}\to\Spec R\to\BTst$ factors
over $\BTst^o$, the ideal $JS$ is a nilideal.
Thus $JS$ is nilpotent as $J$ is finitely generated
by Lemma \ref{Le-BT-o}; hence for sufficiently large $n$ 
we have $\Spec S=\Spec R\times_{\BTst}\Spec A/J^n$.
By construction, $R\to S$ is faithfully flat,
thus injective, and $G\otimes_RS\cong H\otimes_AS$.
Thus $u_R(\Phi_R^1(G))$ preserves $F_1$ since this
holds for $u_{A/J^n}(\Phi_{A/J^n}^1(H))$.
\end{proof}

\begin{Remark}
The above proof of Lemma \ref{Le-Phi-BT} uses than $\Phi_R^1$ 
is an equivalence, but this could be avoided by an more careful
(elementary) analysis of the stack of displays over rings in 
which $p$ is nilpotent. Thus the facts that $\Phi_R^1$ and
$\BT_R^1$ are equivalences can be derived from each other. 
\end{Remark}

Since $\Phi_R^1$ is an equivalence, the isomorphism 
$\Phi_R^1\circ\BT_R^1\cong\id$ of Lemma \ref{Le-Phi-BT} 
induces for each formal $p$-divisible group $G$ over a $R$
an isomorphism
$$
\rho_R(G):G\cong\BT_R^1(\Phi_R^1(G)).
$$

\begin{Thm}
\label{Th-BT-Phi}
For each $p$-divisible group $G$ over a $p$-adic ring $R$ there 
is a unique isomorphism which is functorial in $G$ and $R$ 
$$
\tilde\rho_R(G):\hat G\cong\BT_R(\Phi_R(G))
$$ 
which coincides with $\rho_R(G)$ if $G$ is infinitesimal.
\end{Thm}

If $G$ is an extension of an \'etale $p$-divisible group
by an infinitesimal $p$-divisible group, Theorem \ref{Th-BT-Phi} 
follows from Lemma \ref{Le-BT-Phi-etale} below because both sides of 
$\tilde\rho_R(G)$ preserve short exact sequences.

\begin{Lemma}
\label{Le-BT-Phi-etale}
If $G$ is \'etale, then $\BT_R(\Phi_R(G))$ is zero.
\end{Lemma}

\begin{proof}
Let $\Phi_R(G)=\PPP=(P,Q,F,F_1)$. We have $P=Q$, and $F_1:P\to P$
is an $f$-linear isomorphism. Let $N$ be a nilpotent
$R/p^nR$-algebra for some $n$. By the definition of $\BT_R$,
the group $\BT_R(\PPP)(N)$ is the cokernel of the endomorphism $F_1-1$ 
of $\hat W(N)\otimes_{W(R)}P$. This endomorphism is bijective
because here $F_1$ is nilpotent since $f$ is nilpotent on $\hat W(N)$.
\end{proof}

\begin{proof}[Proof of Theorem \ref{Th-BT-Phi}]
Assume that $G$ is a $p$-divisible group over an $I$-adic 
ring $A$ such that $G_{A/I}$ is infinitesimal. Then
$G_n=G_{A/I^n}$ is infinitesimal as well. Since formal Lie
groups over $A$ are equivalent to compatible systems of formal
Lie groups over $A/I^n$ for $n\ge 1$, the isomorphisms 
$\rho_{A/I^n}(G_n)$ define the desired isomorphism 
$\tilde\rho_A(G)$, which is clearly unique. 
The construction is functorial in the triple $(A,I,G)$.

Assume that in addition a ring homomorphism $u:A\to B$ is
given such that $B$ is $J$-adic and such that $G_{B/J}$
is ordinary; we do note assume that $u(I)\subseteq J$. 
There is a unique exact sequence of $p$-divisible groups 
over $B$
$$
0\to H\xrightarrow\alpha G_B\to H'\to 0
$$ 
such that $H$ is of multiplicative type and $H'$ is \'etale. 
Consider the following diagram of isomorphisms; 
cf.\ Lemma \ref{Le-BT-Phi-etale}.
\begin{equation}
\label{Diag-BT-Phi}
\xymatrix@M+0.2em@C+2em{
\hat H \ar[r]^-{\tilde\rho_B(H)} \ar[d]_\alpha &
\BT_B(\Phi_B(H)) \ar[d]^\alpha \\
\hat G_B \ar[r]^-{(\tilde\rho_A(G))_B} &
(\BT_A(\Phi_A(G)))_B
}
\end{equation}

Since the construction of $\tilde\rho_A(G)$ is functorial in 
$(A,I)$, the diagram commutes if $u(I)\subseteq J$.
If Theorem \ref{Th-BT-Phi} holds, \eqref{Diag-BT-Phi} 
commutes always.
We show directly that \eqref{Diag-BT-Phi} commutes in a special 
case that allows to define $\tilde\rho_R$ in general by descent.
As in the proof of Proposition \ref{Pr-AG-exist} 
we consider a decent
$p$-divisible group $\GG$ over a perfect field $k$ and an affine 
open subscheme $U=\Spf A$ of the isogeny space of $\GG$ over
$\Spf W(k)$. Let $U_{\red}=\Spec A/I$ and let
$G$ be the univeral $p$-divisible group over $A$. By passing
to a connected component of $U$ we may assume that $A$ is integral.
Since $A$ has no $p$-torsion and since $A/pA$ is regular, 
$pA$ is a prime ideal. 
Let $B=\hat A_{pA}$ be the complete local ring of $A$ at this
prime and let $J=pB$.

\begin{Lemma}
\label{Le-BT-Phi}
For this choice of $(A,I,B,J,G)$ the diagram \eqref{Diag-BT-Phi}
is defined and commutes.
\end{Lemma}

\begin{proof}
In order that \eqref{Diag-BT-Phi} is defined we need 
that $G_{B/J}$ is ordinary. Then the defect of commutativity 
of \eqref{Diag-BT-Phi}
is an automorphism $\xi=\xi(A,I,B,J,G)$ of $\hat H$,
which is functorial with respect to $(A,I,B,J,G)$.

Step 1. 
For an arbitrary maximal ideal $\Fm$ of $A$ let
$A_1=\hat A_\Fm$, let $I_1$ be the maximal ideal of
$A_1$, let $B_1$ be the complete local ring of $A_1$
at the prime ideal $pA_1$, and let $J_1=pB_1$.
We have compatible injective homomorphisms
$A\to A_1$ with $I\to I_1$ and $B\to B_1$ with $J\to J_1$.
Moreover $A_1\cong W(k_1)[[t_1,\ldots,t_r]]$ for a finite
extension $k_1$ of $k$, and $G_{A_1}$ is a universal deformation.
Thus $G_{B_1/J_1}$ is ordinary, which implies that $G_{B/J}$ is
ordinary, and it suffices to show that 
$\xi(A_1,I_1,B_1,J_1,G_{A_1})=\id$.

Step 2. 
Next we achieve $r=1$ by a blowing-up construction. 
Let $k_2$ be an algebraic closure of the function field
$k_1(\{t_i/t_j\}_{1\le i,j\le r})$,
let $A_2=W(k_2)[[t]]$, let $I_2=tA_2$, let $B_2$ be the completion 
of $A_2$ at the prime ideal ${pA_2}$, and let $J_2=pB_2$. 
Thus $B_2/J_2=k_2((t))$. There is a natural local homomorphism 
$A_1\to A_2$ with $t_i\mapsto[t_i/t_1]t$, which
induces an injective homomorphism $B_1\to B_2$.
Thus it suffices to show that $\xi(A_2,I_2,B_2,J_2,G_{A_2})=\id$.

Step 3. 
Let $L$ be the completion of an algebraic clocure of 
$k_2((t))$ and let $\OOO\subset L$ be its ring of integers. 
For a fixed $n\ge 1$ let $A'=W_n(\OOO)$ and let $I'$ be the 
kernel of $A'\to\OOO/t\OOO$, thus $I'$ is generated by $([t],p)$.
It is easy to see that $A'$ is $I'$-adic, using that $\OOO$
is $t$-adic and that $p^r[t]=v^r[t^{p^r}]$.
Let $B'=W_n(L)$ and $J'=pB'$. The homomorphism $A_2\to A'$
defined by the inclusion $W(k_2)\to W(L)$ and by $t\mapsto [t]$ 
induces a homomorphism $B_2\to B'$
such that the kernels of $B_2\to B'$ for increasing $n$ have zero
intersection. Thus it suffices to show that
$\xi(A',I',B',J',G_{A'})=\id$.

Step 4.
Since $L$ is algebraically closed, $H_L$ is isomorphic
to $\mu_{p^{\infty}}^d$. Since $G_L$ is ordinary, the 
inclusion $H_L\to G_L$ splits uniquely; i.e.\ we have a 
homomorphism $\psi_L:G_L\to\mu_{p^{\infty}}^d$
that induces an isomorphism of the formal completions.
Since $\OOO$ is normal, the Serre dual of $\psi_L$ extends
to a homomorphism over $\OOO$, thus $\psi_L$
extends to a homomorphism 
$\psi_{\OOO}:G_{\OOO}\to\mu_{p^{\infty}}^d$.
The homomorphism $p^n\psi_{\OOO}$ extends to
$\psi':G\to\mu_{p^{\infty}}^d$ over $A'$, and its
restriction $\psi'_B:G_{B'}\to\mu_{p^{\infty}}^d$ over $B$ 
induces an isogeny of the multiplicative parts, which commutes
with the associated $\xi$'s by functoriality. Thus it suffices 
to show that $\xi'=\xi(A',I',B',J',\mu_{p^{\infty}})=\id$.

Since $A'$ is $p$-adic and since $\mu_{p^\infty}$ is infinitesimal
over $A'/pA'$, the element $\xi''=\xi(A',pA',B',J',\mu_{p^{\infty}})$ 
is well-defined, and it induces $\xi'$ by functoriality. 
But we have $\xi''=\id$ because $A'\to B'$ maps $pA'$ into
$J'=pB'$. This proves Lemma \ref{Le-BT-Phi}.
\end{proof}

We continue the proof of Theorem \ref{Th-BT-Phi}
and write $\BT_R(\Phi_R(G))=G^+$.
Let $\Spec A\to\BTst\times\Spec\ZZ_p$ be the presentation
constructed in the proof of Proposition \ref{Pr-AG-exist}
and let $G$ be the universal group over $A$. 
The ring $A$ is $I$-adic such that $G_{A/I}$ is isoclinic.
Thus the above construction applies and gives 
$\tilde\rho_A(G):\hat G\cong G^+$;
here the components of $\Spec A$ where $G$ is \'etale
do not matter in view of Lemma \ref{Le-BT-Phi-etale}.
Let $\Spec A\times_{\BTst\times\Spec\ZZ_p}\Spec A=\Spec C$
and let $\hat C$ be the $p$-adic completion of $C$.

For an arbitrary ring $R$ in which $p$ is nilpotent 
we want to define $\tilde\rho_R$ by descent, starting
from $\tilde\rho_A(G)$. This is possible if and only if
the inverse images of $\tilde\rho_A(G)$ under the two projections
$p_i:\Spec\hat C\to\Spec A$ coincide, i.e.\ if the following
diagram of formal Lie groups over $\hat C$ commutes, 
where $u:p_1^*G\cong p_2^*G$ is the given descent isomorphism.
\begin{equation}
\label{Diag-descent-BT-Phi}
\xymatrix@M+0.2em{
p_1^*\hat G \ar[r]^{\hat u} \ar[d]_{p_1^*(\tilde\rho)} &
p_2^*\hat G \ar[d]^{p_2^*(\tilde\rho)} \\
p_1^*G^+ \ar[r]^{u^+} & p_2^*G^+ 
}
\end{equation}

Let $A=\prod A_i$ be a maximal decomposition so that
each $A_i$ is a domain, let $B_i$ be the complete
local ring of $A_i$ at the prime $pA_i$, and let
$B=\prod B_i$. Since the two homomorphisms $A\to C$ 
are flat and since $A\to B$ is flat and induces
an injective map $A/p^nA\to B/p^nB$, the natural 
homomorphism $C\to B\otimes_AC\otimes_AB=C'$ induces
an injective homomorphism of the $p$-adic completions
$\hat C\to\hat C'$. Thus the commutativity of
\eqref{Diag-descent-BT-Phi} can be verified over
$\hat C'$. Let $H$ be the multiplicative part of
the ordinary $p$-divisible group $G_B$.

Since the construction of $\tilde\rho$ is functorial with 
respect to the projections of $p$-adic rings
$q_1,q_2:\Spec\hat C'\to\Spec B$,
the following diagram of formal Lie groups over $\hat C'$ commutes.
\begin{equation}
\label{Diag-descent-BT-Phi-1}
\xymatrix@M+0.2em{
q_1^*\hat H \ar[r]^{\hat u} \ar[d]_{q_1^*(\tilde\rho)} &
q_2^*\hat H \ar[d]^{q_2^*(\tilde\rho)} \\
q_1^*H^+ \ar[r]^{u^+} & q_2^*H^+ 
}
\end{equation}
Lemma \ref{Le-BT-Phi} implies that the inclusion $H\to G_B$ 
induces an isomorphism of diagrams \eqref{Diag-descent-BT-Phi-1} 
$\cong$ \eqref{Diag-descent-BT-Phi} $\otimes_{\hat C}\hat C'$.
Thus \eqref{Diag-descent-BT-Phi} commutes as well.
\end{proof}

\subsection{Complement to \cite{Lau-Disp}}

The proof that $\BT_R^1$ is an equivalence in \cite{Lau-Disp}
proceeds along the following lines. First, by \cite{Zink-Disp} 
the functor is always faithful, and fully faithful if $R$ is
reduced over $\FF_p$. Second, by using an $\infty$-smooth 
presentation of $\BTst^o$ as in Lemma \ref{Le-BT-pres}, 
one deduces that $\BT_R^1$ is essentially surjective if $R$
is reduced over $\FF_p$. Using this, one shows that $\BT_R^1$ is
fully faithful in general, and the general equivalence follows.

The second step is based on the following consequence of
the first step. A faithfully flat homomorphism of reduced 
rings $R\to S$ is called an \emph{admissible covering} if
$S\otimes_RS$ is reduced. In this case, a formal
$p$-divisible group $G$ over $R$ lies in the image of
$\BT_R^1$ if $G_S$ lies in the image of $\BT_S^1$.
In \cite[Sec.~3]{Lau-Disp}, some effort is needed 
to find a sufficient supply of admissible coverings.

The proof can be simplified as follows
if one starts with a reduced presentation 
$\pi:\Spec A'\to\BTst\times\Spec\ZZ_p$ with excellent $A'$
such that $A'/\Fm$ is perfect for all maximal ideals $\Fm$ 
of $A'$; see Corollary \ref{Co-small-present}.
Let $\Spec A\to\BTst^o$ be the restriction of $\pi$.
As in \cite{Lau-Disp} it suffices to show that the universal
group $G$ over $A$ lies in the image of $\BT_{A}^1$. 
Since $A$ is excellent, this follows from \cite{Zink-Disp}. 
A direct argument goes as follows: The homomorphisms 
$A\to\prod_{\Fm}A_{\Fm}$ 
and $A_{\Fm}\to\hat A_{\Fm}$ are admissible coverings,
which reduces the surjectivity of $\BT_{A}^1$ to the surjectivity
of $\BT_{A/\Fm^n}^1$. By deformation theory this is reduced
to the case of the perfect fields $A/\Fm$, which is classical.

\subsection{Erratum to \cite{Lau-Disp}}
\label{Subse-Erratum}

\cite[Lemma 3.3]{Lau-Disp} asserts that if $R$ is a noetherian
ring and if $R\to S$ and $S\to T$ are admissible coverings,
then $R\to T$ is an admissible covering too. This is false;
see Example \ref{Ex-Erratum} below. The proof assumes incorrectly 
that a field extension $L/K$ such that $L\otimes_KL$ is reduced 
must be separable. The following part loc.\ cit.\ is proved
correctly.

\begin{Lemma}
\label{Le-Erratum}
Let $R\to S$ be a faithfully flat homomorphism of reduced
rings where $R$ is noetherian such that for all minimal 
prime ideals $\xi\subset S$ and $\eta=\xi\cap R$ the field
extension $R_\eta\to S_\xi$ is separable. Then $S\otimes_RS$
is reduced.
\end{Lemma}

The incorrect \cite[Lemma 3.3]{Lau-Disp} is only used in 
the proof of \cite[Prop.~3.4]{Lau-Disp}, where it can be 
avoided as follows.
For certain rings $A\to\hat A\to\hat B$ one needs that
$\hat B\otimes_A\hat B$ is reduced. The proof shows that 
$A\to\hat A$ and $\hat A\to\hat B$ satisfy
the hypotheses of Lemma \ref{Le-Erratum}. Thus $A\to\hat B$
satisfies these hypotheses as well, and the assertion follows.

\begin{Example}
\label{Ex-Erratum}
Let $K$ be a field of characteristic $p$ and let
$a,b,c$ be part of a $p$-basis of $K^{1/p}$ over $K$.
Let $L=K(X,a+bX)$ and $M=L(Y,a+cY)$ where $X$ and $Y$
are algebraically independent over $K$. Then
$L\otimes_KL$ and $M\otimes_LM$ are reduced, but
$M\otimes_KM$ is not reduced. In particular, $L$ is
not separable over $K$.
\end{Example}



\begin{thebibliography}{BBM}

\bibitem[Be]{Berthelot-Parfait}
P.~Berthelot:
Th\'eorie de Dieudonn\'e sur un anneau de valuation parfait.
Ann.\ Sci.\ Ec.\ Norm.\ Sup.\ (4) {\bf 13} (1980), 225--268

\bibitem[BBM]{BBM} 
P.~Berthelot, L.~Breen, and W.~Messing: 
Th\'{e}orie de Dieudonn\'{e} cristalline II.
Lecture Notes in Math.\ {\bf 930}, Springer Verlag, 1982 

\bibitem[BM]{Berthelot-Messing}
P.~Berthelot, W.~Messing:
Th\'eorie de Dieudonn\'e cristalline, III. 
The Grothendieck Festschrift, Vol.\ I, 173--247, 
Progr.\ Math.\ {\bf 86}, Birkh\"auser, 1990

\bibitem[G1]{Grothendieck-Nice}
A.~Grothendieck:
Groupes de Barsotti-Tate et cristaux. 
Actes du Congr\`es International des Math\'ematiciens (Nice, 1970),
Tome 1, 431--436. 
Gauthier-Villars, Paris, 1971

\bibitem[G2]{Grothendieck-Montreal}
A.~Grothendieck:
Groupes de Barsotti-Tate et Cristaux de Dieudonn\'e.
Universit\'e\ de Montr\'{e}al, 1974

\bibitem[Il1]{Illusie-CC-I}
L.~Illusie: 
Complexe cotangent et d\'eformations, I. 
Lecture Notes in Mathematics, Vol.\ 239,
Springer-Verlag, Berlin-New York, 1971

\bibitem[Il2]{Illusie-BT}
L.~Illusie: 
Deformations de groupes de Barsotti-Tate
(d'apr\`{e}s A.~Grothendieck).
In {\it Seminar on arithmetic bundles: the Mordell conjecture}, 
Ast\'{e}risque {\bf 127} (1985), 151--198 

\bibitem[Ka]{Katz-Slope}
N.~M.~Katz:
Slope filtration of $F$-crystals. 
{\it Journ\'{e}es de G\'{e}om\'{e}trie Alg\'{e}brique de Rennes},
Ast\'{e}risque {\bf 63} (1979), 113--163

\bibitem[Ki]{Kisin-Crys}
M.~Kisin: Crystalline representations and $F$-crystals,
{\em Algebraic geometry and number theory}, 459--496,
Progr.\ Math., Vol.\ 253, Birkh\"auser, 2006

\bibitem[LZ]{Langer-Zink-dRW}
A.~Langer and Th.~Zink:
De Rham-Witt cohomology for a proper and smooth morphism. 
J.\ Inst.\ Math.\ Jussieu {\bf 3} (2004), no.~2, 231--314

\bibitem[La1]{Lau-Disp}
E.~Lau:
Displays and formal $p$-divisible groups.
Invent.\ Math.\ {\bf 171} (2008), 617--628

\bibitem[La2]{Lau-Frames}
E.~Lau:
Frames and finite group schemes over complete
regular local rings.
arXiv:0908.4588,
to appear in Doc.\ Math.\

\bibitem[La3]{Lau-Relation}
E.~Lau:
A relation between crystalline Dieudonn\'e theory
and Dieudonn\'e displays, 2010, available at http://arxiv.org/

\bibitem[Ma]{Matsumura}
H.~Matsumura:
{\it Commutative ring theory}.
Cambridge Univ.\ Press, 1986

\bibitem[Me1]{Messing-Crys}
W.~Messing:
The crystals associated to Barsotti-Tate groups: 
with applications to abelian schemes. 
Lecture Notes in Math.\ {\bf 264}, Springer Verlag, 1972

\bibitem[Me2]{Messing-Disp}
W.~Messing: 
Travaux de Zink.
S\'{e}minaire Bourbaki 2005/2006, exp.\ 964, 
Ast\'{e}risque {\bf 311} (2007), 341--364

\bibitem[NVW]{NVW} 
M.-H. Nicole, A.~Vasiu, and T.~Wedhorn,
Purity of level $m$ stratifications.
arxiv.org:0808.1629,
to appear in Ann.\ Sci.\ Ec.\ Norm.\ Sup.

\bibitem[O1]{Oort-Newton+Formal}
F.~Oort:
Newton polygons and formal groups: 
conjectures by Manin and Grothendieck.
Ann.\ of Math.\ {\bf 152} (2000), 183--206

\bibitem[Po]{Popescu}
D.~Popescu:
General N\'{e}ron desingularization and approximation.
Nagoya Math.\ J.\ {\bf 104} (1986), 85--115

\bibitem[RZ]{Rapoport-Zink}
M.~Rapoport, Th.~Zink:
Period spaces of $p$-divisible groups.
Ann.\ Math.\ Stud.\ {\bf 141}, Princeton Univ.\ Press, 1996

\bibitem[Sw]{Swan-desingu}
R:~Swan: 
N\'{e}ron-Popescu desingularization.
In: Algebra and geometry (Taipei, 1995), 135--192, 
Lect.\ Algebra Geom.\ {\bf 2}, Int.\ Press, Cambridge, MA, 1998 

\bibitem[Va]{Valabrega}
P.~Valabrega:
A few theorems on completion of excellent rings. 
Nagoya Math.\ J.\ {\bf 61} (1976), 127--133 

\bibitem[W]{Wedhorn}
T.~Wedhorn:
The dimension of Oort strata of Shimura varieties of PEL-type. 
In: Moduli of Abelian Varieties, Progr.\ Math.\ Vol.~195,
441--471, Birkh\"{a}user, Basel, 2001

\bibitem[Zi1]{Zink-Disp}
Th.~Zink:
The display of a formal $p$-divisible group.
In: Cohomologies $p$-adiques et applications arithm\'etiques, I, 
Ast\'{e}risque {\bf 278} (2002), 127--248

\end{thebibliography}
\end{document}